\documentclass[12pt]{amsart}

\usepackage[margin=1in]{geometry}
\usepackage{amsmath,amsthm,amssymb}
\usepackage{dsfont}
\usepackage{amsmath,amscd}
\usepackage{mathrsfs}
\usepackage{color}
\usepackage{hyperref}
\usepackage{array}
\usepackage{graphicx}
\usepackage{xcolor}

\hypersetup{
    colorlinks=true, 
    linkcolor=blue, 
    urlcolor=red, 
    linktoc=all 
}

\input xy
\xyoption{all}

\newenvironment{proposition}{\noindent\textbf{Proposition:}\hspace{1em}}{}

\newcommand{\C}{\mathbb{C}}

\newcommand{\PP}{\mathbb{P}}
\newcommand{\Z}{\mathbb{Z}}
\newcommand{\one}{\mathds{1}}

\newcommand{\G}{\mathcal{G}}
\newcommand{\map}{\textrm{Map}}
\newcommand{\sm}[1]{\langle{#1}\rangle}

\theoremstyle{plain}
\newtheorem{thm}{Theorem}[section]

\newtheorem{lemma}[thm]{Lemma}
\newtheorem{remark}[thm]{Remark}

\theoremstyle{definition}

\newcommand{\be}{\begin{equation}}
\newcommand{\ee}{\end{equation}}

\newcommand{\bp}{\begin{proposition}}
\newcommand{\ep}{\end{proposition}}

\newcommand{\ra}{\rightarrow}
\newcommand{\xra}{\xrightarrow}

\newcommand{\dx}{\delta}
\newcommand{\Dx}{\Delta}

\newcommand{\mbz}{\mathbb{Z}}

\newcolumntype{C}[1]{>{\centering\let\newline\\\arraybackslash\hspace{0pt}}m{#1}}

\begin{document}

\title{The Homotopy Type of a Once-Suspended 6-Manifold and Its Applications}
\author[T. Cutler]{Tyrone Cutler}
\address{Fakult\"{a}t F\"{u}r Mathematik, Universit\"{a}t Bielefeld, Universit\"{a}tsstra{\ss}e 25, 33615 Bielefeld, Germany}
\email{tcutler@math.uni-bielefeld.de}

\author[T. So]{Tseleung So}
\address{Department of Mathematics and Statistics, Univeristy of Regina, Regina, SK S4S 0A2, Canada}
\email{tse.leung.so@uregina.ca}


\maketitle

\begin{abstract}
Let $M$ be a closed, oriented, simply connected 6-manifold. After localization away from 2, we give a homotopy decomposition of $\Sigma M$ in terms of spheres, Moore spaces and other recognizable spaces. As applications we calculate generalized cohomology groups of $M$ and determine the homotopy types of gauge groups of certain bundles over $M$. 
\end{abstract}

\section{Introduction}

It is natural to attempt to decompose a given topological space $X$ into a product or union of smaller spaces which are easier to study. In some cases, after suspension, we may obtain a decomposition $\Sigma X\simeq A\vee B$, for spaces $A$ and $B$, called a \emph{suspension splitting} of~$X$. Whenever a suspension splitting exists, the wedge summands $A$ and $B$ are restricted by the (co)homology of $X$. For example, in~\cite{ST19} the second author and Theriault showed that if $M$ is a closed, orientable, smooth, 4-manifold with no 2-torsion in homology, then~$\Sigma M$ is homotopy equivalent to a wedge of spheres, Moore spaces, and $\Sigma\C\PP^2$'s according to $H_*(M)$ and the second Stiefel-Whitney class $w_2(M)$. Applications to cohomology and mapping spaces were found for this splitting result.

In this paper we study the homotopy types of the suspensions of certain 6-manifolds after localization away from 2 and give splitting results analogous to those above. All our manifolds are topological. For those interested in smooth manifolds we note that the PL and smooth categories coincide in dimension 6, so that a closed 6-manifold $M$ is smoothable if and only if its Kirby-Siebenmann invariant $\Dx(M)\in H^4(M;\mbz/2)$ vanishes. Moreover, two smooth 6-manifolds are diffeomorphic if and only if they are homeomorphic~\cite{zubr1988} (although there are smooth 6-manifolds which are tangentially homotopy equivalent but not homeomorphic).

We write $P^n(k)$ for the mapping cone of the degree map $k:S^{n-1}\to S^{n-1}$, and call it an $n$-dimensional mod-$k$ Moore space. We use the symbol $\simeq_{(\frac{1}{2})}$ to denote a homotopy equivalence after localization away from 2. To state our results let $M$ be a closed, oriented, simply connected 6-manifold. By Poincar\'{e} duality it has homology 
\begin{equation}\label{table_original M hmlgy}
\begin{tabular}{C{1.3cm}|C{1.3cm}|C{1.3cm}|C{1.3cm}|C{1.3cm}|C{1.3cm}|C{1.3cm}|C{1.3cm}}
$i$	&$0$	&$1$	&$2$	&$3$	&$4$	&$5$	&$6$\\
\hline
$H_i(M)$	&$\Z$	&$0$	&$\Z^b\oplus T$	&$\Z^{2d}\oplus T$	&$\Z^b$	&$0$	&$\Z$
\end{tabular}
\end{equation}
where $T=\bigoplus^c_{j=1}\Z/p^{r_j}_j$ is a finite abelian group for some primes $p_j$ and integers $r_j\geq1$. Our main theorem identifies the homotopy type of $\Sigma M$ in terms of a wedge of spheres, Moore spaces, and a certain 3-cell complex. 

\begin{thm}\label{main thm_splitting of Sigma M}
Let $M$ be a closed, orientable, simply connected 6-manifold with homology as in~(\ref{table_original M hmlgy}). If the Steenrod power operation $\mathcal{P}^1:H^2(M;\Z/3)\to H^6(M;\Z/3)$ is trivial, then 
\[
\textstyle
\Sigma M\simeq_{(\frac{1}{2})}\bigvee^b_{i=1}(S^3\vee S^5)\vee\bigvee^{2d}_{k=1}S^4\vee\bigvee^c_{j=1}(P^4(p^{r_j}_j)\vee P^5(p^{r_j}_j))\vee S^7.
\]
If $\mathcal{P}^1:H^2(M;\Z/3)\to H^6(M;\Z/3)$ is non-trivial, then either 
\[
\begin{array}{l c}
&
\Sigma M\simeq_{(\frac{1}{2})}\bigvee^{b-1}_{i=1}(S^3\vee S^5)\vee\bigvee^{2d}_{k=1}S^4\vee\bigvee^c_{j=1}(P^4(p^{r_j}_j)\vee P^5(p^{r_j}_j))\vee\Sigma\C\PP^3\\[8pt]
\text{or}\hspace*{1cm}
&
\Sigma M\simeq_{(\frac{1}{2})}\bigvee^b_{i=1}(S^3\vee S^5)\vee\bigvee^{2d}_{k=1}S^4\vee\bigvee_{\substack{1\leq j\leq c \\ j\neq\bar{c}}}P^4(p^{r_j}_j)\vee\bigvee^c_{j=1}P^5(p^{r_j}_j)\vee\Sigma C_{\imath_{\bar{c}}\circ\alpha_1}.
\end{array}
\]
In the last homotopy equivalence $\bar{c}$ is an index such that $p_{\bar{c}}=3$ in $T$ and $C_{\imath_{\bar{c}}\circ\alpha_1}$ is the mapping cone of the composite $\imath_{\bar{c}}\circ\alpha_1:S^6\overset{\alpha_1}{\to}S^3\overset{\imath_{\bar{c}}}{\to}P^4(3^{\bar{c}})$, where $\alpha_1$ is the order 3 map and $\imath_{\bar{c}}$ is the inclusion of the bottom cell. The index $\bar{c}$ is described in the paragraph succeeding Lemma~\ref{lemma_Sigma M' hmtpy type P^1 trivial}.
\end{thm}

\begin{remark}\label{remark_pontryagin = P^1}
If $M$ is smooth, then the condition of the power operation $\mathcal{P}^1:H^i(M;\Z/3)\to H^{i+4}(M;\Z/3)$ being trivial is equivalent to the first Pontryagin class $p_1(M)$ being divisible by 3~\cite{MS}. This gives easily checked geometric critera for the applicability of the theorem. In fact, Jupp~\cite{jupp} has shown that every topological 6-manifold already has a well-defined integral Pontryagin class for which this interpretation is available. 
\end{remark}

Our results should be compared to those of Huang~\cite{huang}, who for $M$ is as in Theorem~\ref{main thm_splitting of Sigma M} gives a decomposition of $\Sigma^2M$ under the additional assumption that $M$ has no 2- or 3-torsion in homology. His splitting makes use of the actions of $Sq^2$, $\mathcal{P}^1$, and of a certain secondary operation $\mathbb{T}$ on the cohomology of $M$, and the additional assumptions allow him to work integrally with this data. While we have the formally weaker localized hypothesis, we make less assumptions on $M$ and obtain our splitting after only a single suspension. This last feature is what grants us the flexibility to extend our applications to the homotopy types of gauge groups  (see Section 5.2).

In Section 5 we give two applications for Theorem~\ref{main thm_splitting of Sigma M}. The first is to calculate the value of any generalized cohomology theory $h^*$ on $M$. For example $h^*$ could be singular cohomology, complex or real $K$-theory, or cobordism. In Section 5.1 we prove the following theorem. The notation $A\cong_{(\frac{1}{2})}B$ for abelian groups $A,B$ means that~\mbox{$A\otimes\mathbb{Z}_{(\frac{1}{2})}\cong B\otimes\mathbb{Z}_{(\frac{1}{2})}$}, where $\mathbb{Z}_{(\frac{1}{2})}=\{\frac{p}{q}\in\mathbb{Q}\mid q\;\text{is a power of}\;2\}$. In particular, if~$A\cong_{(\frac{1}{2})}B$, then $A$ and $B$ are isomorphic up to 2-torsion.

\begin{thm}\label{thm_decomp of reduced cohmlgy thry}
Let $M$ be a closed, orientable, simply connected 6-manifold with homology as in~(\ref{table_original M hmlgy}), and let $h^*$ be a reduced generalized cohomology theory satisfying the wedge axiom. If $\mathcal{P}^1:H^2(M;\Z/3)\to H^6(M;\Z/3)$ is trivial, then for any $n\in\Z$ there is a group isomorphism
\begin{eqnarray*}
h^n(M)
&\cong_{(\frac{1}{2})}&\textstyle\bigoplus^b_{i=1}(h^n(S^2)\oplus h^n(S^4))\oplus\bigoplus^{2d}_{k=1}h^n(S^3)
\oplus\\[8pt]
&&\textstyle\bigoplus^c_{j=1}(h^n(P^3(p^{r_j}_j))\oplus h^n(P^4(p^{r_j}_j)))\oplus h^n(S^6).
\end{eqnarray*}
If $\mathcal{P}^1:H^2(M;\Z/3)\to H^6(M;\Z/3)$ is non-trivial, then either
\begin{eqnarray*}
h^n(M)
&\cong_{(\frac{1}{2})}&\textstyle\bigoplus^{b-1}_{i=1}(h^n(S^2)\oplus h^n(S^4))\oplus\bigoplus^{2d}_{k=1}h^n(S^3\oplus\\[8pt]
&&\textstyle\bigoplus^c_{j=1}(h^n(P^3(p^{r_j}_j))\oplus h^n(P^4(p^{r_j}_j)))\oplus h^n(\C\PP^3)\\[8pt]
\text{or}\hspace*{1cm}h^n(M)
&\cong_{(\frac{1}{2})}&\textstyle\bigoplus^b_{i=1}(h^n(S^2)\oplus h^n(S^4))\oplus\bigoplus^{2d}_{k=1}h^n(S^3)\oplus\bigoplus_{\substack{1\leq j\leq c \\ j\neq\bar{c}}}h^n(P^3(p^{r_j}_j))\oplus\\
&&\textstyle\bigoplus^c_{j=1}h^n(P^4(p^{r_j}_j))\oplus h^n(C_{\imath_{\bar{c}}\circ\alpha_1}).
\end{eqnarray*}
\end{thm}

There is an analogous statement for homology. In particular, similar methods produce obvious decompositions of $h_*(M)$ for any generalised homology theory $h_*$. Interesting applications of this are when $h_*$ is $K$-homology or stable homotopy.

A second application of Theorem~\ref{main thm_splitting of Sigma M} is to study the homotopy types of gauge groups of certain principal bundles over $M$. Given a topological group $G$ and a principal $G$-bundle $P\ra M$, the \emph{gauge group} of $P$ is the topological group consisting of all $G$-equivariant automorphisms of $P$ that fix $M$. Although these objects originated in physics, mathematicians have become interested in their topology because of the connection with geometry. For example, Donaldson Theory~\cite{donaldson1} uses $SU(2)$-gauge theory to construct invariants of smooth 4-manifolds. Another example is the relation between stable bundles over Riemann surfaces and $U(n)$-gauge groups, which was exploited in~\cite{AB83, DU} to compute the cohomology and homotopy groups of moduli spaces of stable bundles.

For a fixed compact connected Lie group $G$ and compact manifold $M$ there are generally infinitely many isomorphism classes of principal $G$-bundles over $M$. Surprisingly, their gauge groups have only finitely many distinct homotopy types~\cite{CS00}. The enumeration of these homotopy types is a problem which has received much attention when $M$ is a 4-dimensional manifold, and there is a long list of works in literature. For example, see~\cite{cutler, HK06,kono91,so16,ST19, theriault17}. However, in the 6-dimensional case the homotopy types of gauge groups are largely unknown. Only a few special cases have been investigated~\cite{HK07,KKT13,KKT14}.

We will apply Theorem~\ref{main thm_splitting of Sigma M} to produce homotopy decompositions of gauge groups of some bundles over $M$. When $P=G\times M$ is the trivial bundle we can determine the homotopy type of its gauge group. The proof will be given in Section 5.2.

\begin{thm}\label{thm_decomp of current grp}
Let $M$ be a closed, orientable, simply connected 6-manifold with homology as in~(\ref{table_original M hmlgy}), and let $G$ be a connected Lie group. Denote the gauge group of the trivial bundle~$G\times M$ by $\G_0(M,G)$. If $\mathcal{P}^1:H^2(M;\Z/3)\to H^6(M;\Z/3)$ is trivial, then
\begin{eqnarray*}
\G_{0}(M,G)
&\simeq_{(\frac{1}{2})}&\textstyle G\times\prod^b_{i=1}(\Omega^2G\times\Omega^4G)\times\prod^{2d}_{k=1}\Omega^3G\times\prod^c_{j=1}(\Omega^3G\{p^{r_j}_j\}\times\Omega^4G\{p^{r_j}_j\})\\
&&\times\Omega^6G,
\end{eqnarray*}
where $\Omega^nG\{m\}=\map^*(P^n(m),G)$. If $\mathcal{P}^1:H^2(M;\Z/3)\to H^6(M;\Z/3)$ is non-trivial, then
\begin{eqnarray*}
\G_{0}(M,G)
&\simeq_{(\frac{1}{2})}&\textstyle G\times\prod^{b-1}_{i=1}(\Omega^2G\times\Omega^4G)\times\prod^{2d}_{k=1}\Omega^3G\times\prod^c_{j=1}(\Omega^3G\{p^{r_j}_j\}G\times\\[6pt]
&&\Omega^4G\{p^{r_j}_j\})
\times\map^*(\C\PP^3,G),\\[6pt]
\text{or}\hspace*{1cm}\G_{0}(M,G)
&\simeq_{(\frac{1}{2})}&\textstyle G\times\prod^b_{i=1}(\Omega^2G\times\Omega^4G)\times\prod^{2d}_{k=1}\Omega^3G\times\prod_{\substack{1\leq j\leq c \\ j\neq\bar{c}}}\Omega^3G\{p^{r_j}_j\}\times\\
&&\textstyle \prod^c_{j=1}\Omega^4G\{p^{r_j}_j\}\times\map^*(C_{\imath_{\bar c}\circ\alpha_1},G).
\end{eqnarray*}
\end{thm}
\noindent We also compute the homotopy types of gauge groups of $SU(n)$-bundles over $M$ which have vanishing second Chern class but are no necessary trivial (see Theorems~\ref{thm_gauge gp general decomp} and~\ref{thm_gauge gp M/Y}).

The structure of this paper is as follows. In Sections 2 and 3 we develop methods that use cohomological information to determine whether certain maps are essential, and apply them to detect elements of $\pi_6(S^4)$. In Section 4 we compute the homotopy type of $\Sigma M$ and prove Theorem~\ref{main thm_splitting of Sigma M}. Section 5 has two subsections. The first of these contains the proof of Theorem~\ref{thm_decomp of reduced cohmlgy thry}, while the second is dedicated to proving Theorem~\ref{thm_decomp of current grp}. Section 5 also contains Theorems~\ref{thm_gauge gp general decomp} and \ref{thm_gauge gp M/Y} and their proofs.

\section*{Acknowledgment}
The second author was supported by Oberwolfach Leibniz Fellowship and is supported by Pacific Institute for the Mathematical Sciences (PIMS) Postdoctoral Fellowship. 

Both authors thank Stephen Theriault for proofreading this paper and giving many helpful comments. We would also like to thank Stefan Behrens, Martin Frankland, and Donald Stanley for useful discussion.

\section{A homotopical test for a null homotopy}
From this point onwards $p$ will denote an odd prime. In this section we generalize Sections 3 and 4 of~\cite{ST19} and develop methods that use (co)homology to detect whether certain maps are null homotopic. Our usage of the term \emph{cofibration sequence} is equivalent to what some authors refer to as a \emph{homotopy cofibration sequence}.
\begin{lemma}\label{lemma_nullity hurewicz}
For $n\geq2$, let $X$ be an $(n-1)$-connected space. If $f:S^n\to X$ is a map such that $f_*:H_n(S^n)\to H_n(X)$ is trivial then $f$ is null homotopic.
\end{lemma}

\begin{proof}
The lemma follows from the Hurewicz Theorem.
\end{proof}

\begin{lemma}\label{lemma_Moore sp hmtpy gps}
For $n\geq 4$, $\pi_n(P^n(p^r)),\pi_n(P^{n-1}(p^r))$ and $[P^n(p^r),S^{n-1}]$ are trivial.
\end{lemma}

\begin{proof}
We may assume all spaces and maps are localized at $p$. Let $m\in\{n-1,n\}$ and let~$f:S^n\to P^m(p^r)$ be a map. Denote by $\rho:P^m(p^r)\to S^m$ the pinch map. Since the composite $S^n\overset{f}{\to}P^m(p^r)\overset{\rho}{\to}S^m$ is null homotopic, $f$ lifts to the homotopy fiber $F$ of~$\rho$. But the $(n+1)$-skeleton of $F$ is homotopy equivalent to $S^{m-1}$, so $f$ factors through a map~$f':S^n\to S^{m-1}$. As $f'$ is null homotopic at $p$, so is $f$.

There is an exact sequence
$\pi_n(S^{n-1})\to[P^n(p^r),S^{n-1}]\to\pi_{n-1}(S^{n-1})\overset{p^r}{\to}\pi_{n-1}(S^{n-1})$.
The first term $\pi_n(S^{n-1})\cong\Z/2$ and the last map is an injection, so $[P^n(p^r),S^{n-1}]$ is trivial.
\end{proof}

\begin{lemma}\label{lemma_P^4-> P^4 v S^3}
Let $f:P^4(p^{r})\to\bigvee^a_{i=1}S^3\vee\bigvee^b_{j=1}P^4(p^{r_j})$ be a map. If $f_*:H_3(P^4(p^{r}))\to H_3(\bigvee^a_{i=1}S^3\vee\bigvee^b_{j=1}P^4(p^{r_j}))$ is trivial, then $f$ is null homotopic.
\end{lemma}

\begin{proof}
By the Hilton-Milnor Theorem and Lemma~\ref{lemma_Moore sp hmtpy gps} we have
\begin{eqnarray*}
\textstyle[P^4(p^{r}),\bigvee^a_{i=1}S^3\vee\bigvee^b_{j=1}P^4(p^{r_j})]
&\cong&\textstyle\bigoplus^a_{i=1}[P^4(p^{r}),S^3]\oplus\bigoplus^b_{j=1}[P^4(p^{r}),P^4(p^{r_j})]\\
&\cong&\textstyle\bigoplus^b_{j=1}[P^4(p^{r}),P^4(p^{r_j})].
\end{eqnarray*}
Thus it suffices to show the composite $f_k:P^4(p^{r})\overset{f}{\to}\bigvee^a_{i=1}S^3\vee\bigvee^b_{j=1}P^4(p^{r_j})\overset{\text{pinch}}{\to}P^4(p^{r_{k}})$ is null homotopic for each $k$. By hypothesis $(f_{k})_*:H_3(P^4(p^{r}))\to H_3(P^4(p^{r_{k}}))$ is trivial. By the Universal Coefficient Theorem $(f_{k})^*:H^4(P^4(p^{r_{k}}))\to H^4(P^4(p^{r}))$ is trivial. Then the nullity of $f_{k}$ follows from~\cite[Theorem~11.7(c)]{neisendorfer}.
\end{proof}

The following lemma is adapted from \cite[Proposition 3.2]{ST19}. We denote $S^n$ by $P^{n+1}(p^{\infty})$.

\begin{lemma}\label{lemma_nullity cup prod}
Fix integers $r,s,m,n$ and $l$ such that
\[
\begin{cases}
\hfil r,s\geq2\quad\text{and}\quad
 n,l\geq3\\
m=n+l\quad\text{or}\quad m=n+l-1.
\end{cases}
\]
Let $t=\min\{r,s\}$, let $f:S^m\to\Sigma P^n(p^r)\wedge P^l(p^s)$ be a map and let $C_{\hat{f}}$ be the mapping cone~of
\[
\hat{f}:S^m\overset{f}{\longrightarrow}\Sigma P^n(p^r)\wedge P^l(p^s)\overset{[\imath_1,\imath_2]}{\longrightarrow}P^{n+1}(p^r)\vee P^{l+1}(p^s),
\]
where $[\imath_1,\imath_2]$ is the Whitehead product of the inclusions $\imath_1:P^{n+1}(p^r)\to P^{n+1}(p^r)\vee P^{l+1}(p^s)$ and~$\imath_2:P^{l+1}(p^s)\to P^{n+1}(p^r)\vee P^{l+1}(p^s)$. Then the following statements are equivalent:
\begin{enumerate}
\item	$f$ is null homotopic;
\item	$f^*:H^*(\Sigma P^n(p^r)\wedge P^l(p^s);\Z/p^{t})\to H^*(S^m;\Z/p^{t})$ is trivial;
\item	all cup products in $\tilde{H}^*(C_{\hat{f}};\Z/p^{t})$ are trivial.
\end{enumerate}
Furthermore, the lemma holds if we allow exactly one of $r,s$ to be $\infty$.
\end{lemma}

\begin{proof}

$(1)\Rightarrow(3)$: If $f$ is null homotopic, then $C_{\hat{f}}\simeq P^{n+1}(p^r)\vee P^{l+1}(p^s)\vee S^{m+1}$ is a suspension and all cup products in $\tilde{H}^*(C_{\hat{f}};\Z/p^t)$ are trivial.

$(3)\Rightarrow(2)$: Consider the diagram of cofibration sequences
\begin{equation}\label{diagram_nullity cup prod compare}
\xymatrix{
S^m\ar[r]^-{f}\ar@{=}[d]	&\Sigma P^n(p^r)\wedge P^l(p^s)\ar[d]^-{[\imath_1,\imath_2]}\ar[r]	&C_f\ar[d]\\
S^m\ar[r]^-{[\imath_1,\imath_2]\circ f}	&P^{n+1}(p^r)\vee P^{l+1}(p^s)\ar[d]\ar[r]	&C_{\hat{f}}\ar[d]^-{a}\\
							&P^{n+1}(p^r)\times P^{l+1}(p^s)\ar@{=}[r]	&P^{n+1}(p^r)\times P^{l+1}(p^s)
}
\end{equation}
where $C_f$ is the mapping cone of $f$ and $a$ is an induced map. The top row implies that $C_f$ is $(n+l-2)$-connected, and the right column implies that
\[
a^*:H^i(P^{n+1}(p^r)\times P^{l+1}(p^s);\Z/p^t)\to H^i(C_{\hat{f}};\Z/p^t)
\]
is an isomorphism for $i\leq\text{max}(n+1,l+1)$.

Apply $H^*(-;\Z/p^t)$ to (\ref{diagram_nullity cup prod compare}) to obtain a diagram with exact rows and columns:
\[
\xymatrix{
H^m(C_f)\ar[r]^-{c}\ar[d]^-{b}	&H^m(\Sigma P^n(p^r)\wedge P^l(p^s))\ar[r]^-{f^*}\ar[d]^-{d}	&H^m(S^m)\\
H^{m+1}(P^{n+1}(p^r)\times P^{l+1}(p^s))\ar@{=}[r]\ar[d]^-{a^*}	&H^{m+1}(P^{n+1}(p^r)\times P^{l+1}(p^s))\ar[d]	&\\
H^{m+1}(C_{\hat{f}})\ar[r]					&H^{m+1}(P^{n+1}(p^r)\vee P^{l+1}(p^s))=0	&
}
\]
Here the coefficients of cohomology are suppressed and $b,c$ and $d$ are induced maps in cohomology. In the left column $H^{m+1}(P^{n+1}(p^r)\times P^{l+1}(p^s))$ is generated by cup products, while all cup products in $H^{m+1}(C_{\hat{f}})$ are trivial by assumption. Thus $a^*$ is the zero map, giving that $b$ is surjective. In the top square, one can check that $d$ is an isomorphism, which implies that $c$ is surjective, and hence by exactness that $f^*$ is the zero map.

$(2)\Rightarrow(1)$: If $m=n+l-1$, then $S^m$ and $\Sigma P^n(p^r)\wedge P^l(p^s)$ are $(m-1)$-connected. By the Universal Coefficient Theorem $f_*:H_m(S^m;\Z)\to H_m(\Sigma P^n(p^r)\wedge P^l(p^s);\Z)$ is the zero map. Lemma~\ref{lemma_nullity hurewicz} then states that $f$ is null homotopic.

Assume $m=n+l$. If $r,s<\infty$, the Hilton-Milnor Theorem and Lemma~\ref{lemma_Moore sp hmtpy gps} imply that
\begin{eqnarray*}
\pi_m(\Sigma P^{n}(p^r)\wedge P^l(p^s))
&\cong&\pi_m(P^m(p^t)\vee P^{m+1}(p^t))\\
&\cong&\pi_m(P^m(p^t))\oplus\pi_m(P^{m+1}(p^t))\\
&\cong&\pi_m(P^{m+1}(p^t)).
\end{eqnarray*}
It follows that $f$ factors through the composite
\[
f':S^m\overset{f}{\to}\Sigma P^n(p^r)\wedge P^l(p^s)\simeq P^m(p^t)\vee P^{m+1}(p^t)\overset{\text{pinch}}{\to}P^{m+1}(p^t).
\]
As $f^*$ is the zero map, so is $(f')^*:H^m(P^{m+1}(p^t);\Z/p^t)\to H^m(S^m;\Z/p^t)$. By the Universal Coefficient Theorem $f'_*:H_m(S^m;\Z)\to H_m(P^{m+1}(p^t);\Z)$ is the zero map, and hence Lemma~\ref{lemma_nullity hurewicz} gives that $f'$ is null homotopic. Of course this shows that $f$ is null homotopic.

On the other hand, if $r=\infty$ and $s<\infty$, then $f\in\pi_m(\Sigma S^{n-1}\wedge P^l(p^s))\cong\pi_m(P^m(p^r))$, which is trivial by Lemma~\ref{lemma_Moore sp hmtpy gps}. The case for $r<s=\infty$ is similar.
\end{proof}

\begin{lemma}\label{lemma_no cup prod}
Suppose $R$ is a commutative ring and $X$ is a space with dimension less than~$n$. Let $f,g:S^n\to X$ be two maps and let $C_f$, $C_g$ and $C_{f+g}$ be the mapping cones of $f,g$ and~$f+g$, respectively. The following statements hold.
\begin{enumerate}
\item If all cup products in $\tilde{H}^*(C_{f};R)$ and $\tilde{H}^*(C_{g};R)$ are trivial, then all cup products in~$\tilde{H}^*(C_{f+g};R)$ are trivial.
\item If all cup products in $\tilde{H}^*(C_{f};R)$ and $\tilde{H}^*(C_{f+g};R)$ are trivial, then all cup products in $\tilde{H}^*(C_{g};R)$ are trivial.
\end{enumerate}
\end{lemma}

\begin{proof}
We start by showing $(1)$. Consider the following diagram
\[
\xymatrix{S^n\vee S^n\ar@{=}[d]\ar[r]^-{f\vee g}&X\vee X\ar[d]^\nabla\ar[r]&C_f\vee C_g\ar[d]^-{\widetilde\nabla}\ar[r]&S^{n+1}\vee S^{n+1}\ar@{=}[d]\\
S^n\vee S^n\ar[r]^-{\{f,g\}}&X\ar@{=}[d]\ar[r]&C_{\{f,g\}}\ar[r]&S^{n+1}\vee S^{n+1}\\
S^n\ar[r]^-{f+g}\ar[u]_-c&X\ar[r]&C_{f+g}\ar[u]_-{\widetilde c}\ar[r]&S^{n+1}\ar[u]_-c}
\]
The rows of the diagram are cofibration sequences. The map $c$ is the suspension comultiplication and $\nabla$ is the folding map. The map $\{f,g\}$ is the composition of $f\vee g$ and $\triangledown$, and $C_{\{f,g\}}$ is its mapping cone. The maps carrying tildes are chosen to make the diagram homotopy commute. Applying $H^*(-;R)$ (and suppressing coefficients) we obtain the following commutative diagram of abelian groups with exact rows
\[
\xymatrix{H^*(S^n)\oplus H^*(S^n)\ar@{=}[d]&\ar[l]_-{f^*\oplus g^*}H^*(X)\oplus H^*(X)&H^*(C_f)\oplus H^*(C_g)\ar[l]&\ar[l]H^*(S^{n+1})\oplus H^*(S^{n+1})\ar@{=}[d]\\
H^*(S^n)\oplus H^*(S^n)\ar[d]^-{c^*}&\ar[l]_-{f^*+ g^*}H^*(X)\ar[u]_{\nabla^*}\ar@{=}[d]&\ar[l]H^*(C_{\{f,g\}})\ar[u]_-{\widetilde\nabla^*}\ar[d]^-{\widetilde c^*}&\ar[l]H^*(S^{n+1})\oplus H^*(S^{n+1})\ar[d]^-{c^*}\\
H^*(S^n)&H^*(X)\ar[l]^-{(f+g)^*}&H^*(C_{f+g})\ar[l]&H^*(S^{n+1})\ar[l].}
\]
A diagram chase shows that $\widetilde\nabla^*$ is injective and $\widetilde c^*$ is surjective. By assumption all cup products in $\tilde{H}^*(C_f;R)$ and $\tilde{H}^*(C_g;R)$ are trivial, so the injectivity of $\widetilde\nabla^*$ shows that all cup products in $\tilde{H}^*(C_{\{f,g\}};R)$ are trivial. On the other hand, each element in $H^*(C_{f+g};R)$ has a preimage in $H^*(C_{\{f,g\}};R)$, so all cup products in $\tilde{H}^*(C_{f+g};R)$ must be trivial.

Part $(2)$ is a consequence of the first statement. The degree $-1$ self-map of $S^n$ induces a homotopy equivalence $C_f\simeq C_{-f}$. Replacing $f$ with $-f$ and $g$ with $f+g$ we have~\mbox{$C_{-f+(f+g)}\simeq C_g$}, and hence obtain the claim.
\end{proof}


\section{Methods of detecting maps in $\pi_6(P^4(p^r))$}

In this section we study $\pi_6(P^4(p^r))$ and consider the problem of detecting its generators. We remind the reader that $p$ denotes an odd prime.

\begin{lemma}\label{lemma_pi_6(P^4)}
If $p=3$, $\pi_6(P^4(p^r))\cong\Z/3^r\oplus\Z/3$. If $p>3$, $\pi_6(P^4(p^r))\cong\Z/p^r$.
In each case the summand $\Z/p^r$ is generated by $[\one,\one]\circ\phi$, where $\phi$ is a generator of~\mbox{$\pi_6(\Sigma P^3(p^r)\wedge P^3(p^r))$} and $[\one,\one]:\Sigma P^3(p^r)\wedge P^3(p^r)\to P^4(p^r)$ is the Whitehead product of the identity map on~$P^4(p^r)$. In the $p=3$ case, the summand $\Z/3$ is generated by $\imath\circ\alpha_1$ where $\alpha_1:S^6\to S^3$ is the order 3 stable map and $\imath:S^3\to P^4(3^r)$ is the inclusion of the bottom cell.

\end{lemma}

\begin{proof}
\cite[Theorem 1.1]{CMN79} shows a homotopy equivalence
\[
\Phi:S^{2n+1}\{p^r\}\times\Omega\textstyle\bigvee^{\infty}_{m=0}P^{4n+2mn+3}(p^r)\stackrel{\simeq}{\longrightarrow}\Omega P^{2n+2}(p^r),
\]
where $S^{2n+1}\{p^r\}$ is the homotopy fiber of degree map $p^r:S^{2n+1}\to S^{2n+1}$. Take $n=1$ to obtain 
$\pi_6(P^4(p^r))\cong\pi_5(S^3\{p^r\})\oplus\pi_6(P^7(p^r))$, where
\[
\begin{array}{c c c}
\pi_6(P^7(p^r))\cong\Z/p^r
&\text{and}
&\pi_5(S^3\{p^r\})\cong\begin{cases}
\Z/3	&p=3\\
0		&p\neq3.
\end{cases}
\end{array}
\]

To identify generators let $\Phi_1$ and $\Phi_2$ be the restrictions of $\Phi$ to $S^3\{p^r\}$ and $\Omega P^7(p^r)$. From the construction of the map $\Phi$ we obtain a commutative diagram
\[
\xymatrix{
\pi_5(P^3(p^r)\wedge P^3(p^r))\ar[r]^-{\sm{E,E}_*}\ar[d]_-{\tilde{E}_*}	&\pi_5(\Omega\Sigma P^3(p^r))\\
\pi_5(\Omega\Sigma P^3(p^r)\wedge P^3(p^r))\ar[ur]_-{(\Phi_2)_*}
}
\]
where $\tilde{E}:P^3(p^r)\wedge P^3(p^r)\to\Omega\Sigma P^3(p^r)\wedge P^3(p^r)$ is the suspension map and $\sm{E,E}$ is the Samelson product of the suspension map $E:P^3(p^r)\to\Omega P^4(p^r)$. By the Freudenthal Suspension Theorem $\tilde{E}_*$ is an isomorphism. Since the adjoint of $\sm{E,E}$ is the Whitehead product $[\one,\one]:\Sigma P^3(p^r)\wedge P^3(p^r)\to P^4(p^r)$ we have $\Phi_2\simeq[\one,\one]\circ\phi$.

Now, there is a commutative diagram with exact rows
\[
\xymatrix{
\pi_6(S^3)\ar[r]^-{p^r}\ar[d]	&\pi_6(S^3)\ar[d]^-{\imath_*}\ar[r]	&\pi_5(S^{3}\{p^r\})\ar[r]\ar[d]^-{(\Phi_1)_*}	&\pi_5(S^{3})\ar[r]^-{p^r}\ar[d]	&\pi_5(S^{3})\ar[d]^-{\imath_*}\\
\pi_6(P^4(p^r))\ar[r]			&\pi_6(P^4(p^r))\ar[r]^-{\cong}		&\pi_5(\Omega P^{4}(p^r))\ar[r]			&0\ar[r]				&\pi_5(P^{4}(p^r))
}
\]
If $p=3$, then the first map in the top row is trivial and the second map is an isomorphism after localization away from 2. Since $\pi_6(S^3)$ is generated by $\alpha_1$, the second square implies~\mbox{$\Phi_1\simeq\imath\circ\alpha_1$}.
\end{proof}

Let us study the generator $\phi\in\pi_6(\Sigma P^3(p^r)\wedge P^3(p^r))$. The Hilton-Milnor Theorem, Lemma~\ref{lemma_Moore sp hmtpy gps} and homotopy equivalence
\begin{equation}\label{eqn_Sigma P^3 wedge P^3}
\Sigma P^3(p^r)\wedge P^3(p^r)\simeq P^6(p^r)\vee P^7(p^r)
\end{equation}
imply $\pi_6(\Sigma P^3(p^r)\wedge P^3(p^r))
\cong\pi_6(P^6(p^r))\oplus\pi_6(P^7(p^r))
\cong\pi_6(P^7(p^r))
\cong\Z/p^r$. Hence $\phi$ is homotopic to the composite
\[
S^6\overset{\jmath}{\longrightarrow}P^7(p^r)\hookrightarrow P^6(p^r)\vee P^7(p^r)\simeq\Sigma P^3(p^r)\wedge P^3(p^r),
\]
where $\jmath$ is the inclusion of the bottom cell of $P^7(p^r)$.

\begin{lemma}\label{lemma_C_kjmath}
For any integer $k$, let $C_{k\jmath}$ be the mapping cone of $k\jmath$ and let $\ell:P^7(p^r)\to C_{k\jmath}$ be the inclusion. There is a cofibration sequence
\[
S^6\overset{k\phi}{\longrightarrow}\Sigma P^3(p^r)\wedge P^3(p^r)\simeq P^6(p^r)\vee P^7(p^r)\overset{1\vee\ell}{\longrightarrow}P^6(p^r)\vee C_{k\jmath}.
\]
If $k=k'p^s$, where $k'$ is coprime to $p$ and $s\geq0$, then $H^6(C_{k\jmath};\Z/p^r)\cong\Z/p^{min(r,s)}$.
\end{lemma}

\begin{proof}
The first part follows from the above discussion and the rest is easily checked.
\end{proof}

Now we will develop methods to detect the generators $[\one,\one]\circ\phi$ and $\imath\circ\alpha_1$. The following lemma says that $[\one,\one]\circ\phi$ can be detected by the existence of non-trivial cup products in the cohomology of its mapping cone. The proof will use two different Whitehead products
\[
\begin{array}{c c c}
[\one,\one]:\Sigma P^3(p^r)\wedge P^3(p^r)\to P^4(p^r)
&\text{and}
&[\imath,\imath]:\Sigma P^3(p^r)\wedge P^3(p^r)\to P^4(p^r)\vee P^4(p^r).
\end{array}\]
The first is the Whitehead product of the identity map on $P^4(p^r)$ with itself, and the second is the Whitehead product of the inclusions $P^4(p^r)\to P^4(p^4)\vee P^4(p^r)$ of $P^4(p^r)$ into the first and the second summands.

\begin{lemma}\label{lemma_mapping cone of b}
For $k\in\Z/p^r$ let $C_k$ be the mapping cone of $k([\one,\one]\circ\phi)$. Then $k=0$ if and only if all cup products in $\tilde{H}^*(C_k;\Z/p^r)$ are trivial. Moreover, for any value $k$ the~$\mathcal{P}^1$-action on $H^*(C_k;\Z/3)$ is trivial.
\end{lemma}

\begin{proof}
To prove the first part of the lemma it suffices to show that $k=0$ if all cup products in $\tilde{H}^*(C_k;\Z/p^r)$ are trivial. Our proof is long so we divide it into three steps. First, label $P^4(p^r)\times*$ and $*\times P^4(p^r)$ in $P^4(p^r)\times P^4(p^r)$ by $P^4_1(p^r)$ and $P^4_2(p^r)$. Hence we write~\mbox{$P^4(p^r)\times P^4(p^r)$} as $P^4_1(p^r)\times P^4_2(p^r)$, and $P^4(p^r)\vee P^4(p^r)$ as $P^4_1(p^r)\vee P^4_2(p^r)$. Let~$C_{[\one,\one]}$ be the mapping cone of $[\one,\one]$. Then there is a diagram of cofibration sequences
\[
\xymatrix{
\Sigma P^3(p^r)\wedge P^3(p^r)\ar[r]^-{[\imath,\imath]}\ar@{=}[d]	&P^4_1(p^r)\vee P^4_2(p^r)\ar[r]\ar[d]^-{\triangledown}	&P^4_1(p^r)\times P^4_2(p^r)\ar[d]^-{\tilde{\triangledown}}\\
\Sigma P^3(p^r)\wedge P^3(p^r)\ar[r]^-{[\one,\one]}					&P^4(p^r)\ar[r]		&C_{[\one,\one]}
}
\]
where $\triangledown$ is the folding map and $\tilde{\triangledown}$ is an induced map. Take cohomology with $\Z/p^r$-coefficient (in the following we suppress the coefficients, which are understood to be $\Z/p^r$ unless otherwise stated) and get the diagram of exact sequences:
\begin{equation}\label{diagram_lemma 3.2 PxP -> C}
\xymatrix{
H^{6}(\Sigma P^3(p^r)\wedge P^3(p^r))\ar[r]^-{\delta}\ar@{=}[d]	&H^7(C_{[\one,\one]})\ar[r]\ar[d]^-{\tilde{\triangledown}^*}	&H^7(P^4(p^r))\ar[d]^-{\triangledown^*}\\
H^{6}(\Sigma P^3(p^r)\wedge P^3(p^r))\ar[r]^-{\delta'}			&H^7(P^4_1(p^r)\times P^4_2(p^r))\ar[r]	&H^7(P^4_1(p^r)\vee P^4_2(p^r))
}
\end{equation}
where $\delta$ and $\delta'$ are connecting maps. Let $u\in H^3(P^4(p^r))$ be a generator and $v\in H^4(P^4(p^r))$ the Bockstein image of $u$. For $i\in\{1,2\}$ let $u_i\in H^3(P^4_i(p^r))$ and $v_i\in H^4(P^4_i(p^r))$ be the images of $u$ and $v$ under the isomorphism $H^*(P^4_i(p^r))\cong H^*(P^4(p^r))$. The K\"{u}nneth Theorem says that
$
H^7(P^4_1(p^r)\times P^4_2(p^r))\cong\Z/p^r\langle u_1\cup v_2,v_1\cup u_2\rangle$. 
As the 5-skeleton of $C_{[\one,\one]}$ is $P^4(p^r)$, we have $H^3(C_{[\one,\one]})\cong\Z/p^r\langle{u}\rangle$ and~\mbox{$H^4(C_{[\one,\one]})\cong\Z/p^r\langle{v}\rangle$}. Since $\tilde{\triangledown}^*(u\cup v)=u_1\cup v_2+v_1\cup u_2$ and $\tilde{\triangledown}^*$ is an isomorphism in degree $7$, $u\cup v$ is a generator of $H^7(C_{[\one,\one]})$. Let $w\in H^7(C_{[\one,\one]})$ be a generator such that $\tilde{\triangledown}^*(w)=u_1\cup v_2-v_1\cup u_2$. Then $H^7(C_{\one,\one})\cong\Z/p^r\sm{u\cup v,w}$. Since $\delta$ in the top row of~(\ref{diagram_lemma 3.2 PxP -> C}) is an isomorphism, there are generators
\begin{equation}\label{eqn_detect k phi, w_6 w_7}
w_6=\delta^{-1}(w),\quad
w_7=\delta^{-1}(u\cup v)\quad
\in H^6(\Sigma P^3(p^r)\wedge P^3(p^r))
\end{equation}
The left square of~(\ref{diagram_lemma 3.2 PxP -> C}) implies $\delta'(w_6)=u_1\cup v_2-v_1\cup u_2$ and $\delta'(w_7)=u_1\cup v_2+v_1\cup u_2$. As~$\delta'$ is an isomorphism, the Bockstein image of $w_6$ is zero while the Bockstein image of $w_7$ is not. Using the homotopy equivalence~(\ref{eqn_Sigma P^3 wedge P^3}), we identify $w_6$ and $w_7$ with the inclusion images of generators of $H^6(P^6(p^r))$ and $H^6(P^7(p^r))$, respectively.

Now let $\jmath:S^6\to P^7(p^r)$ be the inclusion of the bottom cell. Lemma~\ref{lemma_C_kjmath} says that there is a cofibration sequence
$
S^6\overset{k\phi}{\to}\Sigma P^3(p^r)\wedge P^3(p^r)\overset{\tilde{\ell}}{\to}P^6(p^r)\vee C_{k\jmath}$.
Let~$w_k\in H^6(C_{k\jmath})$ be a generator and let $k=k'p^s$ such that $k'$ is coprime to $p$ and $0\leq s\leq r$. Then $w_k$ has order $p^s$,
\begin{equation}\label{eqn_detect k phi, ell}
\tilde{\ell}^*(w_6)=w_6,\quad\text{and}\quad
\tilde{\ell}^*(w_k)=Ap^{r-s}w_7
\end{equation}
for some integer $A$ coprime to $p$.

Next we show that $k=0$ if all cup products in $\tilde{H}^*(C_k)$ are trivial. Consider the diagram of cofibration sequences
\begin{equation}\label{diagram_cofib of k phi}
\begin{gathered}
\xymatrix{
S^6\ar[r]^-{k\phi}\ar@{=}[d]	&\Sigma P^3(p^r)\wedge P^3(p^r)\ar[r]^-{\tilde{\ell}}\ar[d]^-{[\one,\one]}	&C_{k\jmath}\vee P^6(p^r)\ar[d]\\
S^6\ar[r]^-{k([\one,\one]\circ\phi)}	&P^4(p^r)\ar[r]\ar[d]	&C_k\ar[d]^-{g}\\
	&C_{[\one,\one]}\ar@{=}[r]	&C_{[\one,\one]}
}
\end{gathered}
\end{equation}
where $g$ is an induced map. The middle row implies that $H^3(C_k), H^4(C_k), H^7(C_k)$ are $\Z/p^r$, and that $\tilde{H}^i(C_k)$ is zero otherwise. Since $C_{k\jmath}\vee P^6(p^r)$ is 4-connected, the right column implies that $g^*:H^i(C_{[\one,\one]})\to H^i(C_{k})$ is an isomorphism for $i=3$ and $4$. Let $u'\in H^3(C_{k})$ and $v'\in H^4(C_{k})$ be generators such that $g^*(u)=u'$ and $g^*(v)=v'$. 
The right squares of~(\ref{diagram_cofib of k phi}) induce a diagram of exact sequences
\[
\xymatrix{
H^6(C_k)\ar[r]\ar[d]	&H^6(C_{k\jmath}\vee P^6(p^r))\ar[r]^-{\delta''}\ar[d]^-{\tilde{\ell}^*}	&H^7(C_{[\one,\one]})\ar[r]^-{g^*}\ar@{=}[d]	&H^7(C_k)\ar[d]\\
0\ar[r]	&H^6(\Sigma P^3(p^r)\wedge P^3(p^r))\ar[r]^-{\delta}	&H^7(C_{[\one,\one]})\ar[r]	&0
}
\]
where $\delta''$ is a connecting map and $\delta$ is the connecting map in~(\ref{diagram_lemma 3.2 PxP -> C}). 
If all cup products in~$\tilde{H}^*(C_k;\Z/p^r)$ are trivial, then $g^*(u\cup v)=u'\cup v'=0$. The top row implies that
\begin{equation}\label{eqn_detect k phi, w_6, w}
\delta''(Bw_6+Cw_k)=u\cup v
\end{equation}
for some $B,C\in\Z/p^r$. On the other hand, the middle square and equations~(\ref{eqn_detect k phi, w_6 w_7}), (\ref{eqn_detect k phi, ell}) imply
\begin{equation}\label{eqn_detect k phi, w_6, w 2}
\delta''(w_6)=\delta(w_6)=w,\qquad
\delta''(w_k)=\delta(Ap^{r-s}w_7)=Ap^{r-s}u\cup v.
\end{equation}
Combine~(\ref{eqn_detect k phi, w_6, w}) and~(\ref{eqn_detect k phi, w_6, w 2}) to get $ACp^{r-s}\equiv1\pmod{p^r}$. It must be that $r=s$, so $k\equiv0\pmod{p^r}$ and $k([\one,\one]\circ\phi)\simeq*$.

Finally we prove the second part of the lemma. Assume $p=3$. Let $\bar{u}\in H^3(C_{[\one,\one]};\Z/3)$ and~\mbox{$\bar{u}'\in H^3(C_k;\Z/3)$} be the mod-3 images of $u\in H^3(C_{[\one,\one]};\Z/3^r)$ and $u'\in H^3(C_k;\Z/3^r)$. For dimensional reasons it suffices to show that $\mathcal{P}^1(\bar{u}')=0$. Since $\mathcal{P}^1$ acts trivially on $H^*(P^4_1(p^r)\times P^4_2(p^r);\Z/3)$, $\mathcal{P}^1$ also acts trivially on $H^*(C_{[\one,\one]};\Z/3)$. So $\mathcal{P}^1(\bar{u})=0$.

Let $g^*:H^3(C_{[\one,\one]};\Z/3)\to H^3(C_k;\Z/3)$ be the morphism induced by $g$ in the right column of~(\ref{diagram_cofib of k phi}). Then $g^*(\bar{u})=\bar{u}'$. By the naturality of $\mathcal{P}^1$ we have $\mathcal{P}^1(\bar{u}')=g^*(\mathcal{P}^1(\bar{u}))=0$. 
So $\mathcal{P}^1$ acts trivially on $H^*(C_k;\Z/3)$.
\end{proof}

When $p=3$, the group $\pi_6(P^4(3^r))$ has another generator $\imath\circ\alpha_1$. The following lemma says that it can be detected by the $\mathcal{P}^1$-action on the mod-3 cohomology of its mapping cone.

\begin{lemma}\label{lemma_mapping cone of a}
For $p=3$ and $l\in\Z/3$, let $D_l$ be the mapping cone of $l(\imath\circ\alpha_1)$. Then all cup products in $\tilde{H}^*(D_l;\Z/3^r)$ are trivial. Moreover, $l=0$ if and only $\mathcal{P}^1$ acts trivially on~$H^*(D_l;\Z/3)$.
\end{lemma}

\begin{proof}
There is no loss of generality in localizing $D_l$ at $3$. In this case $\alpha_1$ and $\imath$ are co-H-maps, as is $l(\imath\circ\alpha_1)$. Therefore $D_l$ is a co-H-space and all cup products in $\tilde{H}^*(D_l;\Z/3^r)$ are trivial.

For the second statement consider the commutative diagram of cofibration sequences
\[
\xymatrix{
S^6\ar[r]^-{l\alpha_1}\ar@{=}[d]	&S^3\ar[r]\ar[d]^-{\imath}	&D'\ar[d]^-{b}\\
S^6\ar[r]^-{l(\imath\circ\alpha_1)}	&P^4(3^r)\ar[r]				&D_l
}
\]
where $D'$ is the mapping cone of $l\alpha_1$ and $b$ is an induced map. If $l\not\equiv0\pmod{3}$, then $H^*(D';\Z/3)$ has a non-trivial $\mathcal{P}^1$-action. Since $b^*:H^i(D_l;\Z/3)\to H^i(D';\Z/3)$ is isomorphic for $i=3$ and $7$, $\mathcal{P}^1$ acts non-trivially on $H^*(D_l;\Z/3)$.
\end{proof}

\section{The homotopy types of $\Sigma M$}

Let $M$ be a closed, oriented, simply connected 6-manifold with homology as in~(\ref{table_original M hmlgy}). By~\cite[Theorem 0]{jupp} (see also~\cite{wall1966}), $M$ decomposes as a connected sum $M\cong M'\#(\#^d_{i=1}S^3\times S^3)$, where $M'$ is closed, oriented and simply connected, has finite third homology group, and is unique up to oriented homeomorphism. It follows that the homology of $M'$ is
\begin{equation}\label{table_core M hmlgy}
\begin{tabular}{C{1.4cm}|C{1.3cm}|C{1.3cm}|C{1.3cm}|C{1.3cm}|C{1.3cm}|C{1.3cm}|C{1.3cm}}
$i$	&$0$	&$1$	&$2$	&$3$	&$4$	&$5$	&$6$\\
\hline
$H_i(M')$	&$\Z$	&$0$	&$\Z^b\oplus T$	&$T$	&$\Z^b$	&$0$	&$\Z$
\end{tabular}
\end{equation}
where $T=\bigoplus^c_{j=1}\Z/p^{r_j}_j$, with the $p_j$'s prime, is the torsion part of $H_2(M)$. The primes and their exponents may repeat, and we fix once and for all an ordering of them. In the sequel all spaces will be localized away from 2, and for this reason we may assume the $p_j$'s to be odd primes.

We work in the category $\text{CW}_*$ of pointed CW-complexes. All maps will be continuous and preserve basepoints. Fixing minimal CW structures on $M,M'$~\cite[Proposition 4H.3]{hatcher}, each will have a single 0-cell, which we will take as basepoint. Now denote by $M_5$ and $M'_5$ the 5-skeletons of $M,M'$ and let $f:S^5\to M_5$ and $f':S^5\to M'_5$ be the attaching maps of the 6-cells. We have
\begin{equation}\label{equation_M_5=M'_5 v S^3}
\begin{array}{c c c}
M_5\simeq M'_5\vee\bigvee^{2d}_{i=1}S^3
&\text{and}
&f\simeq f'+\omega,
\end{array}
\end{equation}
where $\omega:S^6\to\bigvee^{2d}_{i=1}S^3$ is the attaching map of the top cell in $\#^d_{i=1}S^3\times S^3$. Since $\omega$ is a sum of Whitehead products, $\Sigma\omega$ is null homotopic and we get 
\begin{equation}\label{equation_Sigma M = Sigma M'v S^4}
\textstyle\Sigma M\simeq\Sigma M'\vee\bigvee^{2d}_{i=1}S^4.
\end{equation}
Thus to understand the homotopy type of $\Sigma M$ we should study $\Sigma M'$.

In Section~\ref{section_gauge gp of M}, when working with gauge groups, we will need to construct a particular map $\jmath:Y\to M'_5$, where $Y$ is a 3-dimensional CW-complex. For this reason it will be useful to keep track of a certain subcomplex of $M'_5$ during the proof of Theorem~\ref{main thm_splitting of Sigma M}, and we now establish notation to do this.

 Let $X=\bigvee^b_{i=1}\textbf{S}^2_i\vee\bigvee^c_{j=1}\textbf{P}^3(p^{r_j}_j)$ where $\textbf{S}^2_i=S^2$ and $\textbf{P}^3(p^{r_j}_j)=P^3(p^{r_j}_j)$.
By~\cite[Proposition 4H.3]{hatcher} $M'_5$ can be obtained by two cofibration sequences
\begin{equation}\label{exact seq_M'5 construction}
\begin{array}{c c c}
\bigvee^c_{j=1}P^3(p^{r_j}_j)\overset{g}{\to}X\to C_g
&\text{and}
&\bigvee^b_{i=1}S^3\overset{g'}{\to}C_g\to M'_5
\end{array}
\end{equation}
where $g$ and $g'$ are maps inducing trivial morphisms in homology, and $C_g$ is the mapping cone of $g$.
\begin{lemma}\label{lemma_hmtpy type of Sigma M_5}
Let $M'$ be a closed, simply connected 6-manifold with homology as in~(\ref{table_core M hmlgy}), and let $M'_5$ be its 5-skeleton. Localized away from 2, there is a homotopy equivalence
\[
\textstyle
\Sigma M'_5\simeq_{(\frac{1}{2})}\Sigma X\vee\bigvee^{b}_{i=1}S^5\vee\bigvee^c_{j=1}P^5(p^{r_j}_j).
\]
\end{lemma}
\begin{proof}
Suspending~(\ref{exact seq_M'5 construction}) we get cofibration sequences
\[
\begin{array}{c c c}
\bigvee^c_{j=1}P^4(p^{r_j}_j)\overset{\Sigma g}{\to}\Sigma X\to\Sigma C_g
&\text{and}
&\bigvee^b_{i=1}S^4\overset{\Sigma g'}{\to}\Sigma C_g\to \Sigma M'_5.
\end{array}
\]
The restriction of $\Sigma g$ to each wedge summand $P^4(p^{r_k}_k)$ induces a trivial homomorphism in homology, so is null homotopic by Lemma~\ref{lemma_P^4-> P^4 v S^3}. Thus $\Sigma g$ itself must be null homotopic, and hence $\Sigma C_g\simeq\Sigma X\vee\bigvee^c_{j=1}P^5(p^{r_j}_j).$ Thus $\Sigma g'$ becomes $\Sigma g':\bigvee^b_{i=1}S^4\to\Sigma X\vee\bigvee^c_{j=1}P^5(p^{r_j}_j)$ while
\[
\textstyle
\pi_4(\Sigma X\vee\bigvee^c_{j=1}P^5(p^{r_j}_j))\cong\bigoplus^b_{i=1}\pi_4(\Sigma\textbf{S}^2_i)\oplus\bigoplus^c_{j=1}\left(\pi_4(\Sigma\textbf{P}^3(p^{r_j}_j))\oplus\pi_4(P^5(p^{r_j}_j))\right)
\]
by the Hilton-Milnor Theorem.

Localized away from 2, $\pi_4(\Sigma\textbf{S}^2_i)$ and $\pi_4(\Sigma\textbf{P}^3(p^{r_j}_j))$ are trivial by Lemma~\ref{lemma_Moore sp hmtpy gps}, and thus $\Sigma g'$ factors through $\bigvee^c_{j=1}P^5(p^{r_j}_j)$. For $1\leq i\leq b$ and $1\leq j\leq c$ let $g'_{ij}$ be the composite
\[
\textstyle
g'_{ij}:S^4_i\hookrightarrow\bigvee^b_{i=1}S^4_i\overset{\Sigma g'}{\longrightarrow}\Sigma C_g\overset{\text{pinch}}{\longrightarrow}P^5(p^{r_j}_j).
\]
Since $(g'_{ij})_*:H_4(S^4)\to H_4(P^5(p^{r_j}_j))$ is the zero map, Lemma~\ref{lemma_nullity hurewicz} implies that each $g'_{ij}$ is null homotopic. Therefore $\Sigma g'\simeq*$ and we obtain the asserted homotopy equivalence.
\end{proof}

The Hilton-Milnor Theorem implies that there is an isomorphism
\[
\textstyle
\pi_6(\Sigma M'_5)\cong\bigoplus^{b}_{i=1}(\pi_6(\Sigma\textbf{S}^2_i)\oplus\pi_6(S^5))\oplus\bigoplus^c_{j=1}(\pi_6(\Sigma\textbf{P}^3(p^{r_j}_j))\oplus\pi_6(P^5(p^{r_j}_j)))\oplus W,
\]
where $W$ consists of summands generated by Whitehead products;
\begin{eqnarray*}
W
&\cong&\textstyle\bigoplus_{1\leq i<j\leq b}\pi_6(\Sigma\textbf{S}^2_i\wedge\textbf{S}^2_j)\oplus\bigoplus^b_{i=1}\bigoplus^c_{j=1}(\pi_6(\Sigma\textbf{S}^2_i\wedge\textbf{P}^3(p^{r_j}_j))\oplus\pi_6(\Sigma\textbf{S}^2_i\wedge P^4(p^{r_j}_j)))\\[8pt]
&&\oplus\textstyle\bigoplus_{1\leq j<k\leq c}\pi_6(\Sigma\textbf{P}^3(p^{r_j}_j)\wedge\textbf{P}^3(p^{r_k}_k))\oplus\bigoplus^c_{j,k=1}\pi_6(\Sigma\textbf{P}^3(p^{r_j}_j)\wedge P^4(p^{r_k}_k)).
\end{eqnarray*}
Localized away from 2, Lemma~\ref{lemma_Moore sp hmtpy gps} implies that
\begin{eqnarray}\label{equation_pi_6 Sigma M_5}\nonumber
\pi_6(\Sigma M'_5)
&\cong_{(\frac{1}{2})}&\textstyle\bigoplus^b_{i=1}\pi_6(\Sigma\textbf{S}^2_i)\oplus\bigoplus^c_{j=1}\pi_6(\Sigma\textbf{P}^3(p^{r_j}_j))\oplus\bigoplus^b_{i=1}\bigoplus^c_{j=1}\pi_6(\Sigma\textbf{S}^2_i\wedge P^4(p^{r_j}_j))\oplus\\
&&\textstyle\bigoplus_{1\leq j<k\leq c}\pi_6(\Sigma\textbf{P}^3(p^{r_j}_j)\wedge\textbf{P}^3(p^{r_k}_k))\oplus\bigoplus^c_{j,k=1}\pi_6(\Sigma\textbf{P}^3(p^{r_j}_j)\wedge P^4(p^{r_k}_k)).
\end{eqnarray}
Let $\tilde{f}:S^6\to\Sigma M'_5$ be the attaching map of the 7-cell in $\Sigma M'$. Then
\begin{equation}\label{equation_tilde f Hilton Milnor decomp}
\textstyle
\tilde{f}\simeq_{(\frac{1}{2})}\sum^b_{i=1}x_i+\sum^c_{j=1}y_j+\sum^b_{i=1}\sum^c_{j=1}u_{ij}+\sum_{1\leq j<k\leq c}v_{jk}+\sum^c_{j,k=1}w_{jk},
\end{equation}
where $x_i,y_j,u_{ij},v_{jk}$ and $w_{jk}$ are composites
\[
\begin{array}{l}
x_i:S^6\overset{x'_i}{\to}\Sigma\textbf{S}^2_i\hookrightarrow\Sigma M_5'\\
y_j:S^6\overset{y'_j}{\to}\Sigma\textbf{P}^3(p^{r_j}_j)\hookrightarrow\Sigma M_5'\\
u_{ij}:S^6\overset{u'_{ij}}{\to}\Sigma\textbf{S}^2_i\vee P^5(p^{r_j}_j)\hookrightarrow\Sigma M_5'\\
v_{jk}:S^6\overset{v'_{jk}}{\to}\Sigma\textbf{P}^3(p^{r_j}_j)\vee\Sigma\textbf{P}^3(p^{r_k}_k)\hookrightarrow\Sigma M_5'\\
w_{jk}:S^6\overset{w'_{jk}}{\to}\Sigma\textbf{P}^3(p^{r_j}_j)\vee P^5(p^{r_k}_k)\hookrightarrow\Sigma M_5'
\end{array}
\]
for some maps $x'_i,y'_j,u'_{ij},v'_{jk}$ and $w'_{jk}$. They represent the components of $\tilde{f}$ on the right hand side of~(\ref{equation_pi_6 Sigma M_5}). Now we can apply the lemmas in Section 2 to check whether they are null homotopic or not.

\begin{lemma}\label{lemma_map cone of component no cup prod}
Given a map $h:S^6\to\Sigma M'_5$, let $C_h$ be its mapping cone. If (1) $h=x_i$, (2)~$h=y_j$, (3) $h=u_{ij}$, (4) $h=v_{jk}$, or (5) $h=w_{jk}$, then all cup products in $\tilde{H}^*(C_{h};R)$ are trivial for any principal ideal domain $R$.
\end{lemma}

\begin{proof}
For $h=x_i$, let $h'$ be the composite $h': S^6\overset{\tilde{f}}{\to}\Sigma M_5'\overset{\text{pinch}}{\to}{\Sigma\textbf{S}^2_i}$ and $C_{h'}$ its mapping cone. Since the mapping cone of $\tilde{f}$ is $\Sigma M'$ and the induced morphism $\tilde{f}^*$ in cohomology is the zero map for dimensional reasons, \cite[Lemma 4.2]{ST19} implies that all cup products in $\tilde{H}^*(C_{h'};R)$ are trivial. Notice that $C_h\simeq C_{h'}\vee\bigvee_{\substack{1\leq k\leq b \\ k\neq i}}\Sigma\textbf{S}^2_k\vee\bigvee^{b}_{k=1}S^5\vee\bigvee^c_{j=1}(\Sigma\textbf{P}^3(p^{r_j}_j)\vee P^5(p^{r_j}_j))$, so $\tilde{H}^*(C_h;R)$ has trivial cup products. The $h=y_j$ case can be shown similarly.

For $h=v_{jk}$, let $h''$ be the composite $h'': S^6\overset{f}{\longrightarrow}\Sigma M'_5\overset{\text{pinch}}{\longrightarrow}{\Sigma\textbf{P}^3(p^{r_j}_j)\vee\Sigma\textbf{P}^3(p^{r_k}_k)}\hookrightarrow\Sigma M'_5$ and let $C_{h''}$ be its mapping cone. A similar argument as above shows that all cup products in $\tilde{H}^*(C_{h''};R)$ are trivial. As $h''\simeq_{(\frac{1}{2})}y_j+y_k+v_{jk}$, and we have shown that all cup products in the reduced cohomology ring of the mapping cones of $y_j$ and $y_k$ are trivial, Lemma~\ref{lemma_no cup prod} implies that all cup products in $\tilde{H}^*(C_{h};R)$ are trivial. The case where $h=u_{ij}$ or $w_{jk}$ can be shown similarly.
\end{proof}

\begin{lemma}\label{lemma_attaching map no wh prod}
The components $u_{ij},v_{jk}$ and $w_{jk}$ are null homotopic.
\end{lemma}

\begin{proof}
We show that $u_{ij}$ is null homotopic. It is the composite
\[
S^6\overset{u'}{\longrightarrow}{\Sigma\textbf{S}^2_i}\wedge P^4(p^{r_j}_j)\overset{[\imath_1,\imath_2]}{\longrightarrow}{\Sigma\textbf{S}^2_i}\vee P^5(p^{r_j}_j)
\]
for some map $u'$. Let $C$ be the mapping cone of $u_{ij}$. Lemma~\ref{lemma_map cone of component no cup prod} implies that all cup products in $\tilde{H}^*(C;\Z/p^{r_j}_j)$ are trivial. By Lemma~\ref{lemma_nullity cup prod} $u'$ is null homotopic. It follows that $u_{ij}$ is null homotopic. The triviality of $v_{jk}$ and $w_{jk}$ is established similarly.
\end{proof}

\begin{lemma}\label{lemma_attaching map no b maps}
If $p_j>3$, then $y_j$ is null homotopic. If $p_j=3$, then $y_j$ is the composite
\[
S^6\overset{d\alpha_1}{\longrightarrow}S^3\overset{\imath_j}{\longrightarrow}{\Sigma\textbf{P}^3(p^{r_j}_j)}\hookrightarrow\Sigma M'_5
\]
where $d\in\Z/3$, $\alpha_1:S^6\to S^3$ is the order 3 stable map, and $\imath_j:S^3\to{\Sigma\textbf{P}^3(p^{r_j}_j)}$ is the inclusion of the bottom cell.
\end{lemma}

\begin{proof}
Recall that $y_j$ is the composite $S^6\overset{y'_j}{\to}{\Sigma\textbf{P}^3(p^{r_j}_j)}\hookrightarrow\Sigma M'_5$. Let $C$ be the mapping cone of $y_j$ and $C'$ the mapping cone of $y'_j$. Clearly $C$ is $C'$ wedged with spheres and Moore spaces. We know from Lemma~\ref{lemma_map cone of component no cup prod} that $\tilde{H}^*(C;\Z/p^{r_j}_j)$ has trivial cup products, and thus it follows that $\tilde{H}^*(C';\Z/p^{r_j}_j)$ must also have trivial cup products.


Now, if $p_j>3$, then Lemma~\ref{lemma_pi_6(P^4)} gives $y'_j\simeq k[\one,\one]\circ\phi$, and Lemma~\ref{lemma_mapping cone of b} then implies $k=0$. If $p_j=3$, then Lemma~\ref{lemma_pi_6(P^4)} gives $y'_j\simeq k([\one,\one]\circ\phi)+l(\imath_j\circ\alpha_1)$ for some $k\in\Z/3^{r_j}\Z$ and $l\in\Z/3$. Let $C_k$ and $D_l$ be the mapping cones of $k([\one,\one]\circ\phi)$ and $l(\imath_j\circ\alpha_1)$. As all cup products in~$\tilde{H}^*(D_l;\Z/p^{r_j}_j)$ are trivial by Lemma~\ref{lemma_mapping cone of a}, all cup products in $\tilde{H}^*(C_k;\Z/p^{r_j}_j)$ are trivial by Lemma~\ref{lemma_no cup prod}. Therefore Lemma~\ref{lemma_mapping cone of b} implies $k=0$.
\end{proof}

Lemmas~\ref{lemma_attaching map no wh prod} and~\ref{lemma_attaching map no b maps} imply that $\Sigma M'\simeq_{(\frac{1}{2})}(\Sigma X\cup e^7)\vee\bigvee^b_{i=1}S^5\vee\bigvee^c_{j=1}P^5(p^{r_j}_j)$ and the non-trivial components of the attaching map $\tilde{f}$ can only be $x_i$'s and $y_j$'s for $p_j=3$. Lemma~\ref{lemma_mapping cone of a} tells us that these maps are detected by $\mathcal{P}^1$. Hence if $\mathcal{P}^1$ acts trivially on $H^*(M';\Z/3)$, then the attaching map $\tilde{f}_{(\frac{1}{2})}$ is null homotopic and we obtain the following lemma.
\begin{lemma}\label{lemma_Sigma M' hmtpy type P^1 trivial}
Suppose $\mathcal{P}^1$ acts trivially on $H^*(M';\Z/3)$. Then
\[
\pushQED{\qed}
\textstyle
\Sigma M'\simeq_{(\frac{1}{2})}\bigvee^b_{i=1}(\Sigma\textbf{S}^2_i\vee S^5)\vee\bigvee^c_{j=1}(\Sigma\textbf{P}^3(p^{r_j}_j)\vee P^5(p^{r_j}_j))\vee S^7. \qedhere
\popQED
\]
\end{lemma}

Suppose $\tilde{f}_{(\frac{1}{2})}$ is not homotopic to the constant map. Then some of the $x_i$'s and $y_j$'s in~(\ref{equation_tilde f Hilton Milnor decomp}) are non-trivial. Let $\bar{b}$ and $\bar{c}$ be the numbers of non-trivial $x_i$'s and $y_j$'s. Notice that if $y_{j}$ is non-trivial then $p_j=3$. Relabeling the indices, we assume that $x_1,\ldots,x_{\bar{b}}$ and $y_1,\ldots,y_{\bar{c}}$ are non-trivial and $r_1\leq \cdots\leq r_{\bar{c}}$. Then the homotopy equivalence~(\ref{equation_tilde f Hilton Milnor decomp}) becomes
\begin{equation}
\label{def of c bar}
\textstyle\tilde{f}\simeq_{(\frac{1}{2})}\sum^{\bar{b}}_{i=1}x_{i}+\sum^{\bar{c}}_{j=1}y_{j},
\end{equation}
where the $x_i$'s and $y_j$'s are composites
\[
x_{i}:S^6\overset{A_i\alpha_1}{\longrightarrow}{\Sigma\textbf{S}^2_{i}}\hookrightarrow\Sigma M_5',\qquad
y_{j}:S^6\overset{B_j\alpha_1}{\longrightarrow}S^3\overset{\imath_{j}}{\longrightarrow}{\Sigma\textbf{P}^3(p^{r_{j}}_{j})}\hookrightarrow\Sigma M_5'
\]
for $A_i,B_j\in\{1,-1\}$. We are going to simplify the expression of $\tilde{f}$.

\begin{lemma}\label{lemma_construction of hmtpy equiv}
Let $\tilde{\varphi}:X\to X$ be a homotopy equivalence. Then there exists a CW-complex $N$ and a homotopy equivalence $\varphi:M'\to N$ such that $N$ is formed by attaching 3-cells, 4-cells and one 6-cell to $X$ and the restriction of $\varphi$ to $X$ is the composite
\[
X\overset{\tilde{\varphi}}{\longrightarrow}X\hookrightarrow N.
\]
\end{lemma}

\begin{proof}
First form the following two diagrams
\[
\xymatrix{
\bigvee^c_{j=1}P^3(p^{r_j}_j)\ar[r]^-{g}\ar@{=}[d]				&X\ar[r]\ar[d]^-{\tilde{\varphi}}	&C_{g}\ar[d]^-{\varphi_1}\\
\bigvee^c_{j=1}P^3(p^{r_j}_j)\ar[r]^-{\tilde{\varphi}\circ g}	&X\ar[r]							&C_{\tilde{\varphi}\circ g}
}\quad
\xymatrix{
\bigvee^b_{i=1}S^3\ar[r]^-{g'}\ar@{=}[d]		&C_g\ar[r]\ar[d]^-{\varphi_1}			&M'_5\ar[d]^-{\varphi_2}\\
\bigvee^b_{i=1}S^3\ar[r]^-{\varphi_1\circ g'}	&C_{\tilde{\varphi}\circ g}\ar[r]		&C_{\varphi_1\circ g'}
}
\]
where the top rows are the cofibration sequences in~(\ref{exact seq_M'5 construction}), $\varphi_1$ and $\varphi_2$ are induced maps, and $C_{\varphi'\circ g}$ and $C_{\varphi_1\circ g'}$ are mapping cones of $\varphi'\circ g$ and $\varphi_1\circ g'$. Since $\tilde\varphi$ is a homotopy equivalence, so is $\varphi_1$, and hence also $\varphi_2$. 

Next let $f':S^5\to M'_5$ be the attaching map of the top cell in $M'$ and let $N$ be the mapping cone of $\varphi_2\circ f$. Then we have the following diagram
\[
\xymatrix{
S^5\ar[r]^-{f'}\ar@{=}[d]		&M'_5\ar[r]\ar[d]^-{\varphi_2}	&M'\ar[d]^-{\varphi}\\
S^5\ar[r]^-{\varphi_2\circ f'}	&C_{\varphi_1\circ g'}\ar[r]	&N
}
\]
where the rows are cofibration sequences and $\varphi$ is an induced map. Again $\varphi$ is a homotopy equivalence. By construction $N$ is formed by attaching 3-cells, 4-cells and one 6-cell to $X$, and the restriction of $\varphi$ to $X$ is the composite of $\tilde\varphi:X\to X$ and the inclusion $X\hookrightarrow N$.
\end{proof}

\begin{lemma}\label{lemma_Sigma M' hmtpy type P^1 non-trivial}
Suppose $\mathcal{P}^1$ acts non-trivially on $H^*(M';\Z/3)$. Then there exists a homotopy equivalence $\varphi:M'\to N$, where $N$ is formed by attaching 3-cells, 4-cells and one 6-cell to $X$. If $\bar{b}\geq 1$, then $N$ can be chosen so that
\[
\textstyle
\Sigma N\simeq_{(\frac{1}{2})}\bigvee^{b}_{i=2}\Sigma\textbf{S}^2_i\vee\bigvee^b_{i'=1}S^5\vee\bigvee^c_{j=1}(\Sigma\textbf{P}^3(p^{r_j}_j)\vee P^5(p^{r_j}_j))\vee\Sigma C_{\alpha_1}.
\]
If $\bar{b}=0$ and $\bar{c}\geq 1$, then $C$ can be chosen so that
\[
\textstyle
\Sigma N\simeq_{(\frac{1}{2})}\bigvee^b_{i=1}(\Sigma\textbf{S}^2_i\vee S^5)\vee\bigvee^c_{j=2}\Sigma\textbf{P}^3(p^{r_j}_j)\vee\bigvee^c_{j=1}P^5(p^{r_j}_j)\vee\Sigma C_{\imath_{\bar{c}}\circ\alpha_1}.
\]
Here $C_{\alpha_1}$ is the mapping cone of $\alpha_1:S^6\to\Sigma\textbf{S}^2_1$ and $C_{\imath_{\bar{c}}\circ\alpha_1}$ is the mapping cone of the composite $S^6\overset{\alpha_1}{\to}S^3\overset{\imath_{\bar{c}}}{\to}\Sigma\textbf{P}^3_{\bar{c}}(3^{r_{\bar{c}}})$. 
\end{lemma}

\begin{proof}
Let $\bar{b}\geq1$. Given a homotopy equivalence $\tilde{\varphi}:X\to X$,  Lemma~\ref{lemma_construction of hmtpy equiv} gives a suitable $N$ and a homotopy equivalence $\varphi:M'\to N$ whose restriction to $X$ is $X\overset{\tilde{\varphi}}{\to}X\hookrightarrow N$. As $M'$ and $N$ are homotopy equivalent, Lemmas~\ref{lemma_hmtpy type of Sigma M_5}, \ref{lemma_map cone of component no cup prod}, \ref{lemma_attaching map no wh prod} and~\ref{lemma_attaching map no wh prod} imply that $\Sigma N\simeq_{(\frac{1}{2})}(\Sigma X\cup e^7)\vee\bigvee^b_{i=1}S^5\vee\bigvee^c_{j=1}P^5(p^{r_j}_j)$, where $e^7$ attaches to $\Sigma\textbf{S}^2_i$ and $\Sigma\textbf{P}^3(3^{r_j})$ possibly by $\alpha_1$ and $\imath_j\circ\alpha_1$ that are detected by $\mathcal{P}^1$. Thus we will be done if we can construct a suitable $\tilde{\varphi}$ such that $e^7$ attaches only to $\Sigma\textbf{S}^2_1$. 

To this end, for $2\leq i\leq b$, let $h_i$ be the composite
\[
\textbf{S}^2_1\vee\textbf{S}^2_i\overset{\text{comult}\vee 1}{\longrightarrow}\textbf{S}^2_1\vee\textbf{S}^2_1\vee\textbf{S}^2_i\overset{1\vee(-A_iA_1+1)}{\longrightarrow}\textbf{S}^2_1\vee\textbf{S}^2_i
\]
on $\textbf{S}^2_1\vee\textbf{S}^2_i$ and the identity on the other wedge summands of $X$. For $1\leq j\leq\bar{c}$, let $\imath_j:\textbf{S}^2_1\to\textbf{P}^3(3^{r_j})$ be the inclusion of the bottom cell and let $h'_j$ be the composite
\[
\textbf{S}^2_1\vee \textbf{P}^3(3^{r_j})\overset{\text{comult}\vee 1}{\longrightarrow}\textbf{S}^2_1\vee \textbf{S}^2_1\vee \textbf{P}^3(3^{r_j})\overset{1\vee(-B_jA_1\imath_j+1)}{\longrightarrow}\textbf{S}^2_1\vee \textbf{P}^3(3^{r_j})
\]
on $\textbf{S}^2_1\vee\textbf{P}^3(3^{r_j})$ and the identity on the other wedge summands of $X$. Define $\tilde{\varphi}:X\to X$ to be the composite $\tilde{\varphi}=h_2\circ\ldots\circ h_{\bar{b}}\circ h'_1\circ\ldots\circ h'_{\bar{c}}$. For $1\leq i\leq b$, $1\leq j\leq c$ let~$\mu_i,\nu_j\in H^2(X;\Z/3)$  be the duals of the homology classes defined by the inclusions of~$\textbf{S}^2_i$ and the bottom cell of $\textbf{P}^3(p^{r_j}_j)$. Then $\tilde{\varphi}^*(\mu_i)=-A_1A_i\mu_1+\mu_i$ for $2\leq i\leq\bar{b}$, $\tilde{\varphi}^*(\nu_j)=-A_1B_j\mu_1+\nu_j$ for $1\leq j\leq\bar{c}$, and $\tilde{\varphi}^*(\mu_i)=\mu_i$ and $\tilde{\varphi}^*(\nu_j)=\nu_j$ for other cases. Using the naturality of $\mathcal{P}^1$ one can show that $e^7$ only attaches to $\Sigma\textbf{S}^2_1$ in $\Sigma N$. By the previous remarks we are complete.

Now let $\bar{b}=0$ and $\bar{c}\geq1$. Define a homotopy equivalence $\tilde{\varphi}:X\to X$ as follows. For $1\leq j<\bar{c}$ let $\tilde{\imath}_j:P^3(3^{r_{\bar{c}}})\to P^3(3^{r_j})$ be a map defined by the diagram of cofibration sequences
\[
\xymatrix{
S^2\ar[r]^-{3^{r_{\bar{c}}}}\ar[d]_-{3^{r_{\bar{c}}-r_j}}	&S^2\ar[r]\ar@{=}[d]	&P^3(3^{r_{\bar{c}}})\ar[d]^-{\tilde{\imath}_j}\\
S^2\ar[r]^-{3^{r_j}}								&S^2\ar[r]				&P^3(3^{r_j}),
}
\]
and let $h''_{j}:X\to X$ be the composite
\begin{eqnarray*}
\textbf{P}^3(3^{r_{\bar{c}}})\vee\textbf{P}^3(3^{r_j})
&\xrightarrow{\text{comult}\vee\one}&\textbf{P}^3(3^{r_{\bar{c}}})\vee\textbf{P}^3(3^{r_{\bar{c}}})\vee\textbf{P}^3(3^{r_j})\\
&\xrightarrow{\one\vee(-B_jB_{\bar{c}}\tilde{\imath}_j)\vee\one}&\textbf{P}^3(3^{r_{\bar{c}}})\vee\textbf{P}^3(3^{r_j})\vee\textbf{P}^3(3^{r_j})\\
&\xrightarrow{\one\vee\text{fold}}&\textbf{P}^3(3^{r_{\bar{c}}})\vee\textbf{P}^3(3^{r_j})
\end{eqnarray*}
on $\textbf{P}^3(3^{r_{\bar{c}}})\vee\textbf{P}^3(3^{r_j})$ and the identity on the other wedge summands of $X$. Define $\tilde{\varphi}:X\to X$ to be $\tilde{\varphi}=h''_1\circ\ldots\circ h''_{\bar{c}-1}$. Then $\tilde{\varphi}^*(\nu_j)=-B_1B_j\mu_{\bar{c}}+\nu_j$ for $1\leq j\leq\bar{c}-1$, $\tilde{\varphi}^*(\nu_j)=\nu^*_j$ for other cases and $\tilde{\varphi}^*(\mu_i)=\mu_i$ for $1\leq i\leq b$. Using the naturality of $\mathcal{P}^1$ one can show that $e^7$ attaches only to $\Sigma\textbf{P}^3(3^{r_{\bar{c}}})$ in $\Sigma N$. So we obtain the asserted homotopy equivalence.
\end{proof}


\begin{proof}[Proof of Theorem~\ref{main thm_splitting of Sigma M}]
By~(\ref{equation_Sigma M = Sigma M'v S^4}) $\mathcal{P}^1$ acts trivially on $H^*(M;\Z/3)$ if and only if $\mathcal{P}^1$ acts trivially on $H^*(M';\Z/3)$. Combining Lemmas~\ref{lemma_Sigma M' hmtpy type P^1 trivial} and~\ref{lemma_Sigma M' hmtpy type P^1 non-trivial} we obtain the theorem.
\end{proof}

\section{Applications}

This section contains applications for Theorem~\ref{main thm_splitting of Sigma M}. We show first how to decompose $h_*(M)$ for any generalized homology theory $h_*$ and record our findings as Theorem~\ref{thm_decomp of reduced cohmlgy thry}. After this we determine the homotopy types of certain gauge groups of principal bundles over $M$. A proof of Theorem~\ref{thm_decomp of current grp} is given in this section. Experts will be interested in the Theorems \ref{thm_gauge gp general decomp} and \ref{thm_gauge gp M/Y} which we also prove.

\subsection{Cohomology theories}

For our first application let $h^*$ be a generalized cohomology theory. For instance $h^*$ could be real or complex K-theory, or cobordism. Applying it to $M$ we may use Theorem~\ref{main thm_splitting of Sigma M} to decompose $h^*(M)$ up to 2-torsion. For technical reasons we use reduced cohomology, which we assume to satisfy the wedge axiom. A general reference for the subject is Hatcher's book~\cite[\S3.1]{hatcher} .

\begin{proof}[Proof of Theorem~\ref{thm_decomp of reduced cohmlgy thry}]
According to the Brown Representability Theorem there is an $\Omega$-spectrum $E=\{E_*\}$ such that $h^n(X)\cong [X,E_n]$ for any CW complex $X$ (see~\cite[\S4E]{hatcher}). The group structure on the homotopy set on the right-hand side comes from the infinite loop space structure of the component $E_n$. 

Now, the homotopical localization map $\phi:E_n\rightarrow (E_n)_{(\frac{1}{2})}$ is a loop map, so induces a homorphism $\phi_*:[X,E_n]\rightarrow[X,(E_n)_{(\frac{1}{2})}]$. When $X$ is a simply connected finite CW complex, $\phi_*$ is an algebraic localization \cite[Chapter 2, Corollary 6.5]{hiltonmislinroitberg}. Therefore in this case $h^n(X)\otimes\mathbb{Z}_{(\frac{1}{2})}\cong [X, (E_n)_{(\frac{1}{2})}]$, and moreover $[X,(E_n)_{(\frac{1}{2})}]\cong [X_{(\frac{1}{2})}, (E_n)_{(\frac{1}{2})}]$.

To apply this it will be enough to take $X=\Sigma M$, and since $(\Sigma M)_{\frac{1}{2}}\simeq \Sigma (M_{\frac{1}{2}})$, the group $[\Sigma M_{(\frac{1}{2})}, (E_n)_{(\frac{1}{2})}]$ spilts into a direct sum by virtue of the wedge decomposition given in Theorem~\ref{main thm_splitting of Sigma M}. The same reasoning applies also when $X$ is one of the

\end{proof}

\subsection{Gauge groups of principal bundles over $M$}\label{section_gauge gp of M}

Let $G$ be a topological group and $P$ be a principal $G$-bundle over $M$ induced by a pointed map $\gamma:M\to BG$. Then the isomorphism class of $P$ is classified by the unpointed homotopy class of $\gamma$. The \emph{gauge group} of~$P$, denoted $\G_{\gamma}(M,G)$, is the topological group consisting of all $G$-equivariant automorphisms of $P$ that fix $M$. Let $\map^*_{\gamma}(M,BG)$ and $\map_{\gamma}(M,BG)$ be the connected components of $\map^*(M,BG)$ and $\map(M,BG)$, respectively, that contain $\gamma$. By~\cite{AB83} there is a homotopy equivalence $B\G_{\gamma}(M,G)\simeq\map_{\gamma}(M,BG)$. We will apply Theorem~\ref{main thm_splitting of Sigma M} to determine the homotopy types of $\G_{\gamma}(M,G)$ for certain principal bundles $P$. In the sequel $G$ will denote a simple, simply connected, compact Lie group.

\subsubsection{Gauge groups of trivial bundles}
The trivial bundle $P=G\times M$ is induced by the constant map. Its gauge group $\G_{0}(M,G)$ is homeomorphic to $\map(M,G)$~\cite[Chapter 7.2]{husemoller}. We can use the following lemma and Theorem~\ref{main thm_splitting of Sigma M} to prove Theorem~\ref{thm_decomp of current grp}.

\begin{lemma}
Let $X$ be a space such that $\Sigma X\simeq\Sigma A\vee\Sigma B$ for some spaces $A$ and $B$. Then
\[
\map(X,G)\simeq G\times\map^*(A,G)\times\map^*(B,G).
\]
\end{lemma}

\begin{proof}
Consider the fibration $\map^*(X,G)\to\map(X,G)\overset{ev}{\to}G$, where $ev$ evaluates a map at the base point. The canonical inclusion $G\to\map(X,G)$ induces a homotopy equivalence $\map(X,G)\simeq G\times\map^*(X,G)$ and the lemma then follows from the sequence
\begin{eqnarray*}
\map^*(X,G)
&\simeq&\map^*(\Sigma X,BG)\\
&\simeq&\map^*(\Sigma A,BG)\times\map^*(\Sigma B,BG)\\
&\simeq&\map^*(A,G)\times\map^*(B,G).
\end{eqnarray*}
\end{proof}

\subsubsection{Gauge groups of $SU(n)$-bundles over $M$}
If $P$ is not the trivial bundle, then the gauge group of $P$ is difficult to compute in general. In this section we focus on gauge groups of principal $SU(n)$-bundles over $M$ whose second Chern classes are trivial.

As $SU(n)$ is connected, $BSU(n)$ is simply connected, and the pointed and unpointed homotopy classes of a given map $\gamma:M\to BSU(n)$ agree. Hence there is a one-to-one correspondence between the isomorphism classes of principal $SU(n)$-bundles over $M$ and members of the pointed homotopy set $[M,BSU(n)]$.

\begin{lemma}\label{lemma_classification of SU-bundle}
Let $A$ be a 6-dimensional CW-complex with $H^6(A)$ torsion free. Then for~\mbox{$n\geq3$}, there is a bijection $[A,BSU(n)]\cong H^4(A)\times H^6(A)$. In particular, if $P$ is a principal $SU(n)$-bundle over $A$ with second Chern class $c_2(P)$ and third Chern class $c_3(P)$, then the isomorphism class of $P$ is determined by the pair $(c_2(P),c_3(P))$.
\end{lemma}

\begin{proof}
Since $A$ is a 6-complex, to compute $[A,BSU(n)]$ it suffices to compute the homotopy classes of maps from $A$ into the $6^{\text{th}}$ Postnikov section $Z_6$ of $BSU(n)$. The low dimensional homotopy groups of $SU(n)$ for $n\geq3$ are well known, and since $\pi_{i+1}(BSU(n))\cong \pi_i(SU(n))$ for all $i\geq0$, an easy calculation shows that $Z_6$ sits in a principal fibration sequence
\[
\dots\longrightarrow K(\Z,3)\overset{\delta Sq^2}{\longrightarrow} K(\Z,6)\longrightarrow Z_6\longrightarrow K(\Z,4)\overset{\delta Sq^2}{\longrightarrow} K(\Z,7)
\]
where $\dx:K(\Z_2,6)\ra K(\Z,7)$ is the Bockstein connecting map. Apply $[A,-]$ to it and get the exact sequence of pointed sets
\[
\dots\ra H^3(A)\overset{\delta Sq^2}{\longrightarrow}H^6(A)\longrightarrow[A,Z_6]\longrightarrow H^4(A)\overset{\delta Sq^2}{\longrightarrow}H^7(A)
\]
Here $[A,K(\Z,4)]\cong H^4(A)$, and in this context we can understand the isomorphism as being induced by sending a map $f:A\to BSU(n)$ to $f^*(c_2)\in H^4(A)$. Since $H^7(A)=0$, any map~$A\to K(\Z,4)$ lifts to $Z_6$. 

Now, homotopy classes of lifts differ by the action of $[A,K(\Z,6)]\cong H^6(A)$. Rutter~\cite{JR} shows how to compute the isotropy groups of this action away from the constant map. Following his Theorem 1.3.1 we study the cohomology operation $\dx Sq^2:H^3(A)\ra H^6(A)$. Our assumptions ensure that this operation acts trivially, since it factors through $H^5(A;\Z_2)$, and $H^6(A)$ is torsion free. With the identifications above, we get that $[A,Z_6]$ is the disjoint union of copies of $H^6(A)$, one copy for each element of $H^4(A)$. The lemma follows.
\end{proof}

For $n\geq 3$, suppose $P$ is a principal $SU(n)$-bundle over $M$ with
\[
\begin{array}{c c c}
c_2(P)=(0,\ldots,0)\in\Z^b\oplus\bigoplus^c_{j=1}\Z/p^{r_j}_j
&\text{and}
&c_3(P)=l\in\Z.
\end{array}
\]
Denote the gauge group of $P$ by $\G_{\vec{0},l}(M,SU(n))$. We compute the homotopy type of $\G_{\vec{0},l}(M,SU(n))$ in two steps. The first, based on the idea of~\cite[Theorem 2.4]{so16}, is to find a 3-dimensional CW-complex $Y$ and a map $\jmath:Y\to M$ such that
\begin{equation}\label{equation_gauge gp decomposition}
\G_{\vec{0},l}(M,SU(n))\simeq\G_{\vec{0},l}(C_{\jmath},SU(n))\times\map^*(Y,SU(n)).
\end{equation}
where $C_{\jmath}$ is the mapping cone of $\jmath$. The second step, after making sense of the right-hand side of~\eqref{equation_gauge gp decomposition}, is to analyse the homotopy type of $C_{\jmath}$ and decompose $\G_{\vec{0},l}(C_{\jmath},SU(n))$ into a product of recognizable spaces.

Observe that both $M$ and $C_{\jmath}$ meet the criteria of Lemma~\ref{lemma_classification of SU-bundle} and hence we have
\[
[M,BSU(n)]\cong H^4(M)\times H^6(M)\quad
\text{and}\quad
[C_{\jmath},BSU(n)]\cong H^4(C_{\jmath})\times H^6(C_{\jmath})
\]
for $n\geq3$. A key assumption in the sequel will be that the suspension $\Sigma \jmath:\Sigma Y\rightarrow \Sigma M$ has a left homotopy inverse. This assumption implies that~$\Sigma M\simeq\Sigma Y\vee\Sigma C_{\jmath}$, and hence that the inclusion map $q:M\ra C_{\jmath}$ induces bijections on $H^4$ and $H^6$. In this way we obtain bijections
\begin{equation*}
q^*:[C_{\jmath},BSU(n)]\ra [M,BSU(n)],\qquad n\geq3.
\end{equation*}
The upshot is that in this case the path components of $\map(M,BSU(n))$ and $\map(C_{\jmath},BSU(n))$ are in correspondence and the labelling in equation \eqref{equation_gauge gp decomposition} makes sense.

\begin{lemma}\label{lemma_coaction of M/Y}
Let $\jmath:Y\to M$ be a map from a 3-dimensional CW-complex $Y$. Then there is an induced coaction $\psi':C_{\jmath}\to C_{\jmath}\vee S^6$ such that diagrams
\[
\begin{array}{c c c c}
\xymatrix{
M\ar[d]^-{q}\ar[r]^-{\psi}	&M\vee S^6\ar[d]^-{q\vee id}\\
C_{\jmath}\ar[r]^-{\psi'}			&C_{\jmath}\vee S^6
}
&\xymatrix{
C_{\jmath}\ar[d]^-{p}\ar[r]^-{\psi'}&C_{\jmath}\vee S^6\ar[d]^-{p\vee 1}\\
S^6\ar[r]^-{\sigma}	&S^6\vee S^6
}
&\xymatrix{
C_{\jmath}\ar[d]^-{\psi'}\ar[r]^-{\psi'}	&C_{\jmath}\vee S^6\ar[d]^-{\psi'\vee id}\\
C_{\jmath}\vee S^6\ar[r]^-{id\vee\sigma}	&C_{\jmath}\vee S^6\vee S^6
}
\end{array}
\]
are homotopy commutative, where $\psi:M\to M\vee S^6$ is the coaction associated with the minimal cell decomposition of $M$, $p:C_{\jmath}\to S^6$ is the pinch map, and $\sigma:S^6\to S^6\vee S^6$ is the comultiplication.
\end{lemma}
\begin{proof}
Because $Y$ is 3-dimensional, the map $\jmath$ factors through a map $\jmath':Y\to M_5$. If $f:S^5\ra M_5$ is the attaching map for the top cell of $M$, then there is an induced homotopy commutative diagram 
\[
\xymatrix{&S^5\ar[d]_-{f}\ar@{=}[r]&S^5\ar[d]^-{f'}\\
Y\ar[r]^-{\jmath'}\ar@{=}[d]&M_5\ar[r]^-{q'}\ar[d]&C_{\jmath'}\ar[d]\\
Y\ar[r]^-{\jmath}&M\ar[r]^-q&C_{\jmath}}
\]
in which $C_{\jmath'}$ is the mapping cone of $\jmath'$, $q'$ is the inclusion map and $f'$ is the composite of $f$ and $q'$, the bottom right square is a homotopy pushout, and the rows and colums are cofibre sequences. The coaction $\psi'$ is that associated with the right-hand column, and its interaction with $\psi$ (that is, the commutativity of the first square in the proposition statement) is granted by the diagram. The second and third diagrams of the proposition statement follow from the standard properties of coactions~\cite{A01}.
\end{proof}

\begin{lemma}\label{lemma_connecting map null hmtpic}
Assume that $\Sigma \jmath:\Sigma Y\to\Sigma M$ has a left homotopy inverse. Then writing $\vec{0}=(0,\ldots,0)$, for any $l\in\Z$ there is a homotopy fibration sequence
\begin{equation}\label{mapspacehomotopfibseq}
\map^*(Y,SU(n))\overset{\delta_l}{\to}\map^*_{\vec{0},l}(C_{\jmath},BSU(n))\overset{q^*}{\to}\map^*_{\vec{0},l}(M,BSU(n))
\end{equation}
in which the map $\delta_l$ is null homotopic.
\end{lemma}

\begin{proof}
Consider the induced map $q^*:Map^*(C_{\jmath},BSU(n))\ra Map^*(M,BSU(n))$. It's well known that the homotopy fibre of $q^*$ over the constant map is homotopy equivalent to~$Map^*(\Sigma Y,BSU(n))\simeq Map^*(Y,SU(n))$. We need to identify the homotopy fibres of $q^*$ over the other maps. 

We have the coaction $\psi':C_{\jmath}\ra C_{\jmath}\vee S^6$ from Lemma~\ref{lemma_coaction of M/Y}. Identifying 
\[
Map^*(C_{\jmath},BSU(n))\times Map^*(S^6,BSU(n))\cong Map^*(C_{\jmath}\vee S^6,BSU(n))
\]
we get a map
\begin{equation*}
\theta':Map^*(C_{\jmath},BSU(n))\times Map^*(S^6,BSU(n))\ra Map^*(C_{\jmath},BSU(n))
\end{equation*}
which is induced by $\psi'$. The fact that $\psi'$ is a coaction implies that $\theta'$ has the properties of a homotopy action. Similar reasoning converts the coaction $\psi:M\ra M\vee S^6$ into to a homotopy action
\begin{equation*}
\theta:Map^*(M,BSU(n))\times Map^*(S^6,BSU(n))\ra Map^*(M,BSU(n)).
\end{equation*}

Now, in terms of path components we have
\[
\pi_0(Map^*(C_{\jmath},BSU(n)))\cong[C_{\jmath},BSU(n)]\cong H^4(C_{\jmath})\times H^6(C_{\jmath})
\]
by the the discussion before Lemma~\ref{lemma_coaction of M/Y}, and $\pi_0(Map^*(S^6,BSU(n)))\cong H^6(S^6)\cong\mathbb{Z}$. The fact that the induced map
\begin{equation*}
\psi^{'*} :H^6(C_{\jmath})\oplus H^6(S^6)\cong H^6(C_{\jmath}\vee S^6)\xrightarrow{}H^6(C_{\jmath})
\end{equation*}
acts as addition $\mathbb{Z}\oplus\mathbb{Z}\rightarrow \mathbb{Z}$ shows that $\theta'$ restricts to components as a map
\begin{equation*}
\theta':Map^*_{(\vec{0},l)}(C_{\jmath},BSU(n))\times Map^*_m(S^6,BSU(n))\ra Map^*_{(\vec{0},l+m)}(C_{\jmath},BSU(n)),
\end{equation*}
where $l,m\in\mathbb{Z}$. The restriction of $\theta$ to components is analogous.

Next, fixing a map $f_{-l}\in Map^*_{-l}(S^6,BSU(n))$, let 
\[
\begin{array}{c}
\theta_l=\theta(-,f_{-l}):\map^*_{\vec{0},l}(M,BSU(n))\to\map^*_{\vec{0},0}(M,BSU(n))\\[8pt]
\theta'_l=\theta'(-,f_{-l}):\map^*_{\vec{0},l}(C_{\jmath},BSU(n))\to\map^*_{\vec{0},0}(C_{\jmath},BSU(n))
\end{array}
\]
be the induced maps. These are homotopy equivalences and fit in the following diagram
\begin{equation}\label{diagram_htpy pullback}
\begin{gathered}
\xymatrix{Map^*_{(\vec{0},l)}(C_{\jmath},BSU(n))\ar[r]_-\simeq^-{\theta'_l}\ar[d]_-{q^*}&Map^*_{(\vec{0},0)}(C_{\jmath},BSU(n))\ar[d]_-{q^*}\\
Map^*_{(\vec{0},l)}(M,BSU(n))\ar[r]_-\simeq^-{\theta_l}&Map^*_{(\vec{0},0)}(M,BSU(n)).}
\end{gathered}
\end{equation}
The fact that this diagram homotopy commutes is the content of the first homotopy commutativity square of Lemma~\ref{lemma_coaction of M/Y}. Since the horizontal maps are homotopy equivalences, the diagram is a homotopy pullback. It follows that the horizontal arrows induce a homotopy equivalence between the homotopy fibres of the vertical map. In this way we get the fibration sequence~\eqref{mapspacehomotopfibseq}. To complete we need only show that $\delta_l$ is null homotopic.

For this we extend the homotopy pullback~(\ref{diagram_htpy pullback}) by adding the homotopy fibres of the vertical maps to get
\begin{equation}\label{diagram_dlfact}
\begin{gathered}
\xymatrix{
\map^*(Y,SU(n))\ar[r]^-{\delta_l}\ar[d]_-\simeq	&\map^*_{\vec{0},l}(C_{\jmath},BSU(n))\ar[r]^-{q^*}\ar[d]^-{\theta_l'}_\simeq	&\map^*_{\vec{0},l}(M,BSU(n))\ar[d]^-{\theta_l}_\simeq\\
\map^*(Y,SU(n))\ar[r]^-{\delta_0}				&\map^*_{\vec{0},0}(C_{\jmath},BSU(n))\ar[r]^-{q^*}&\map^*_{\vec{0},0}(M,BSU(n)).
}
\end{gathered}
\end{equation}
The bottom sequence is induced by the maps in the cofibration sequence
\begin{equation*}
Y\xra{\jmath} M\xra{q} C_{\jmath}\xra{\dx} \Sigma Y\xra{\Sigma\jmath}\Sigma M\ra\dots
\end{equation*}
Since $\Sigma\jmath$ has a left homotopy inverse, the map $\delta$ is null homotopic. Since $\delta_0$ is induced by $\delta$, it too is null homotopic, and hence the factorization in (\ref{diagram_dlfact}) show that $\delta_l$ is null homotopic.
\end{proof}

Finally we prove the homotopy equivalence~(\ref{equation_gauge gp decomposition}).

\begin{lemma}\label{lemma_gauge decomp for dim <4}
Assume the notation and hypothesis of Lemma~\ref{lemma_connecting map null hmtpic}. For any $l\in\Z$ there is a homotopy equivalence
$
\G_{\vec{0},l}(M,SU(n))\simeq\G_{\vec{0},l}(C_{\jmath},SU(n))\times\map^*(Y,SU(n))$.
\end{lemma}
\begin{proof}
There is a diagram of evaluation fibration sequences
\[\xymatrix{
SU(n)\ar[r]\ar@{=}[d]	&\map^*_{\vec{0},l}(C_{\jmath},BSU(n))\ar[r]\ar[d]^-{q^*}	&\map_{\vec{0},l}(C_{\jmath},BSU(n))\ar[r]^-{ev}\ar[d]^-{q^*}	&BSU(n)\ar@{=}[d]\\
SU(n)\ar[r]^-{\partial}				&\map^*_{\vec{0},l}(M,BSU(n))\ar[r]		&\map_{\vec{0},l}(M,BSU(n))\ar[r]^-{ev}		&BSU(n)
}\]
where the two $q^*$'s are induced by $q$, and $\partial$ is a connecting map for the bottom fibration sequence. Expand the left square to get a diagram of homotopy fibration sequences
\[
\xymatrix{
	&\Omega\map^*_{\vec{0},l}(M,BSU(n))\ar@{=}[r]\ar[d]		&\Omega\map^*_{\vec{0},l}(M,BSU(n))\ar[d]^-{\jmath^*}\\
\G_{\vec{0},l}(C_{\jmath},SU(n))\ar[r]^-{q^*}\ar@{=}[d]	&\G_{\vec{0},l}(M,SU(n))\ar[r]^-{a}\ar[d]	&\map^*(Y,SU(n))\ar[d]^-{\delta_l}\\
\G_{\vec{0},l}(C_{\jmath},SU(n))\ar[r]			&SU(n)\ar[r]\ar[d]_-{\partial}							&\map^*_{\vec{0},l}(C_{\jmath},BSU(n))\ar[d]^-{q^*}\\
	&\map^*_{\vec{0},l}(M,BSU(n))\ar@{=}[r]		&\map^*_{\vec{0},l}(M,BSU(n))
}\]
where $\jmath^*$ is induced by $\jmath$, $a$ is an induced map, and the right column is given by Lemma~\ref{lemma_connecting map null hmtpic}. The middle right square of the diagram is a homotopy pullback, so since $\delta_l$ is null homotopic, $a$ has a homotopy right inverse. This gives the asserted homotopy equivalence.
\end{proof}

Of course, to make use of the previous lemmas we must construct the map $\jmath:Y\ra M$, and we turn now to this. We require that $\Sigma\jmath$ have a homotopy left inverse, which will happen if and only if $\Sigma M\simeq\Sigma Y\vee \Sigma C_{\jmath}$. Theorem~\ref{main thm_splitting of Sigma M} says that after suspension and localization away from 2 the homotopy type of $\Sigma M$ is given by one of the following three cases.
\begin{equation}
\label{equation_Sigma M compare}
\begin{array}{l c}
\Sigma M\simeq_{(\frac{1}{2})}\bigvee^{b}_{i=1}\Sigma\textbf{S}^2_i\vee\bigvee^b_{i=1}S^5\vee\bigvee^{2d}_{k=1}S^4\vee\bigvee^c_{j=1}(\Sigma\textbf{P}^3(p^{r_j}_j)\vee P^5(p^{r_j}_j))\vee S^7	&\text{Case A}\\[10pt]
\Sigma M\simeq_{(\frac{1}{2})}\bigvee^{b}_{i=2}\Sigma\textbf{S}^2_i\vee\bigvee^b_{i=1}S^5\vee\bigvee^{2d}_{k=1}S^4\vee\bigvee^c_{j=1}(\Sigma\textbf{P}^3(p^{r_j}_j)\vee P^5(p^{r_j}_j))\vee\Sigma C_{\alpha_1}	&\text{Case B}\\[10pt]
\Sigma M\simeq_{(\frac{1}{2})}\bigvee^b_{i=1}(\Sigma\textbf{S}^2_i\vee S^5)\vee\bigvee^{2d}_{k=1}S^4\vee\bigvee_{\substack{1\leq j\leq c \\ j\neq\bar{c}}}\Sigma\textbf{P}^3(p^{r_j}_j)\vee\bigvee^c_{j=1}P^5(p^{r_j}_j)\vee\Sigma C_{\imath_{\bar{c}}\circ\alpha_1}	&\text{Case C}
\end{array}
\end{equation}
Be careful in Cases B and C, because the inclusion of $\Sigma\textbf{S}^2_i$ into the wedge sum need not be homotopic to the suspension of the inclusion $\textbf{S}^2_i\hookrightarrow M$. This is because the above homotopy equivalence of $\Sigma M$ is given by the homotopy equivalence $\Sigma\varphi:\Sigma M\to\Sigma N$ produced in Lemma~\ref{lemma_Sigma M' hmtpy type P^1 non-trivial}. The same happens for $\Sigma\textbf{P}^3(p^{r_j}_j)$.

\begin{lemma}\label{lemma_existence of j}
There exist a 3-dimensional CW-complex $Y$ and a map $\jmath:Y\to M$ such that $\Sigma\jmath$ has a left homotopy inverse. More precisely, $Y$ can be chosen to be
\begin{itemize}
\item	$\bigvee^{b}_{i=1}\textbf{S}^2_i\vee\bigvee^c_{j=1}\textbf{P}^3(p^{r_j}_j)\vee\bigvee^{2d}_{k=1}S^3$ in Case A;\\
\item	$\bigvee^{b}_{i=2}\textbf{S}^2_i\vee\bigvee^c_{j=1}\textbf{P}^3(p^{r_j}_j)\vee\bigvee^{2d}_{k=1}S^3$ in Case B;\\
\item	$\bigvee^{b}_{i=1}\textbf{S}^2_i\vee\bigvee_{\substack{1\leq j\leq c \\ j\neq\bar{c}}}\textbf{P}^3(p^{r_j}_j)\vee\bigvee^{2d}_{k=1}S^3$ in Case C.
\end{itemize}
\end{lemma}

\begin{proof}
In Case A let $\jmath$ be the inclusion $X\vee\bigvee^{2d}_{k=1}S^3\hookrightarrow M$. Then $\Sigma\jmath$ is the inclusion of $\bigvee^b_{i=1}\Sigma\textbf{S}^2_i\vee\bigvee^c_{j=1}\Sigma\textbf{P}^3(p^{r_j}_j)\vee\bigvee^{2d}_{k=1}S^4$ into the wedge sum in~(\ref{equation_Sigma M compare}) and hence has a left homotopy inverse.

In Case B we use notation of Lemma~\ref{lemma_Sigma M' hmtpy type P^1 non-trivial}. Let $\jmath'$ be the composite
\[\textstyle
\jmath':X\vee\bigvee^{2d}_{k=1}S^3\hookrightarrow N\vee\bigvee^{2d}_{k=1}S^3\overset{\varphi'\vee 1}{\longrightarrow}M'\vee\bigvee^{2d}_{k=1}S^3\hookrightarrow M,
\]
where $\varphi'$ is the homotopy inverse to $\varphi$, and let $\jmath$ be the composite
\[\textstyle
\jmath:\bigvee^{b}_{i=2}\textbf{S}^2_i\vee\bigvee^c_{j=1}\textbf{P}^3(p^{r_j}_j)\vee\bigvee^{2d}_{k=1}S^3\hookrightarrow X\vee\bigvee^{2d}_{k=1}S^3\overset{\jmath'}{\longrightarrow}M.
\]
Then $\Sigma\jmath$ is homotopic to the inclusion of $\bigvee^{b}_{i=2}\Sigma\textbf{S}^2_i\vee\bigvee^c_{j=1}\Sigma\textbf{P}^3(p^{r_j}_j)\vee\bigvee^{2d}_{k=1}S^4$ into the wedge sum in~(\ref{equation_Sigma M compare}) and hence has a left homotopy inverse. Similarly in Case C let $\jmath$ be the composite
\[\textstyle
\jmath:\bigvee^{b}_{i=1}\textbf{S}^2_i\vee\bigvee_{\substack{1\leq j\leq c \\ j\neq\bar{c}}}\textbf{P}^3(p^{r_j}_j)\vee\bigvee^{2d}_{k=1}S^3\hookrightarrow X\vee\bigvee^{2d}_{k=1}S^3\overset{\jmath'}{\longrightarrow}M.
\]
Again $\Sigma\jmath$ has a left homotopy inverse.
\end{proof}

\begin{thm}\label{thm_gauge gp general decomp}
Let $M$ be a simply connected, closed, orientable 6-manifold with homology as in~(\ref{table_original M hmlgy}), and let $P$ be a principal $SU(n)$-bundle over $M$ with $c_2(P)=(0,\ldots,0)$ and $c_3(P)=l$. Let $Y$ and $\jmath:Y\to M$ be defined as in Lemma~\ref{lemma_existence of j}, and let $C_{\jmath}$ be the mapping cone of $\jmath$. Suppose $\mathcal{P}^1:H^2(M;\Z/3)\to H^6(M;\Z/3)$ is trivial (Case A). Then
\begin{eqnarray*}
\G_{\vec{0},l}(M,SU(n))
&\simeq_{(\frac{1}{2})}&\textstyle\G_{\vec{0},l}(C_{\jmath},SU(n))\times\prod^{b}_{i=1}\Omega^2SU(n)\times\prod^{c}_{j=1}\Omega^3\{p^{r_j}_j\}SU(n)\times\\
&&\textstyle\prod^{2d}_{k=1}\Omega^3SU(n).
\end{eqnarray*}
Suppose $\mathcal{P}^1:H^2(M;\Z/3)\to H^6(M;\Z/3)$ is non-trivial. Then in Case B we have
\begin{eqnarray*}
\G_{\vec{0},l}(M,SU(n))
&\simeq_{(\frac{1}{2})}&\textstyle\G_{\vec{0},l}(C_{\jmath},SU(n))\times\prod^{b}_{i=2}\Omega^2SU(n)\times\prod^{c}_{j=1}\Omega^3\{p^{r_j}_j\}SU(n)\times\\
&&\textstyle\prod^{2d}_{k=1}\Omega^3SU(n),
\end{eqnarray*}
whilst in Case C we have
\begin{eqnarray*}
\G_{\vec{0},l}(M,SU(n))
&\simeq_{(\frac{1}{2})}&\textstyle\G_{\vec{0},l}(C_{\jmath},SU(n))\times\prod^{b}_{i=1}\Omega^2SU(n)\times\prod_{\substack{1\leq j\leq c \\ j\neq\bar{c}}}\Omega^3\{p^{r_j}_j\}SU(n)\times\\
&&\textstyle\prod^{2d}_{k=1}\Omega^3SU(n).
\end{eqnarray*}
\end{thm}


\begin{proof}
The claims follow immediately from Lemmas~\ref{lemma_gauge decomp for dim <4} and~\ref{lemma_existence of j}.
\end{proof}

We want to further decompose the term $\G_{\vec{0},l}(C_{\jmath},SU(n))$ appearing in Theorem~\ref{thm_gauge gp general decomp}, but in general this is difficult. If the torsion group $T$ in~(\ref{table_original M hmlgy}) is trivial, or is a cyclic group, then we can show that $C_{\jmath}$ is a wedge of a CW-complex and some 4-spheres. In this case we proceed via the following two lemmas.

\begin{lemma}\label{lemma_mapping cone simplified}
Let $A$ be a simply connected space and let $B=A\cup\bigcup^{\ell}_{i=1}e^{n}$ be formed by attaching $\ell$ many $n$-cells to $A$ for $n,{\ell}\geq2$. If $\pi_{n-1}(A)$ is a cyclic group, then $B\simeq B'\vee\bigvee^{\ell-1}_{i=1}S^n$ where $B'=A\cup e^n$ is formed by attaching one~$n$-cell to $A$.
\end{lemma}

\begin{proof}
The lemma is obvious if $\pi_{n-1}(B)$ is the trivial group.

Suppose $\pi_{n-1}(A)\cong\Z/m$, where $m$ is zero or an integer greater than 1. Here $\Z/m$ means~$\Z$ if $m=0$. Label the $n$-cells in $B$ by $e^n_1,\ldots,e^n_{\ell}$ and their boundaries by $S^{n-1}_1,\ldots,S^{n-1}_{\ell}$. Let~$g:\bigvee^{\ell}_{i=1}S^{n-1}_i\to A$ be the attaching map of the $n$-cells, let $g_i:S^{n-1}_i\to A$ be the restriction of $g$ to $S^{n-1}_{i}$, and let $\epsilon:S^{n-1}\to A$ be a map representing a generator of $\pi_{n-1}(A)$. Then $g_i\simeq a_i\epsilon$ for some $a_i\in\Z/m$. Define a homomorphism
\[
\Theta:[\textstyle{\bigvee^{\ell}_{i=1}}S^{n-1}_i,A]\to(\Z/m)^{\ell} 
\]
by $\Theta(h)=(b_1,\ldots,b_{\ell})$ for $h\in[\bigvee^{\ell}_{i=1}S^{n-1}_i,A]$, where the restriction of $h$ to $S^{n-1}_i$ is homotopic to $b_i\epsilon$. Then $\Theta$ is an isomorphism satisfying $\Theta(g)=(a_1,\ldots,a_{\ell})$.

Now, there is an invertible matrix $P\in\text{GL}_{\ell}(\Z)$ with $(a_1,\ldots,a_{\ell})\cdot P\equiv(a',0,\ldots,0)$, where
\[
a'=\begin{cases}
\text{gcd}\{a_i|1\leq i\leq {\ell}\}	&m=0\\[5pt]
\text{gcd}\{m,\tilde{a}_i|1\leq i\leq {\ell}\}	&m\geq2
\end{cases}
\]
and $\tilde{a}_i$ is an integer such that $\tilde{a}_i\equiv a_i\pmod{m}$. As there are group isomorphisms
\[
[\textstyle\bigvee^{\ell}_{i=1}S^{n-1}_i,\bigvee^{\ell}_{j=1}S^{n-1}_j]
\cong\bigoplus^{\ell}_{i,j=1}[S^{n-1}_i,S^{n-1}_j]
\cong\text{Mat}_{\ell}(\Z),
\]
there exists a homotopy equivalence $\Phi:\bigvee^{\ell}_{i=1}S^{n-1}_i\to\bigvee^{\ell}_{i=1}S^{n-1}_i$ such that
\[
\Theta(g\circ\Phi)=(a_1,\ldots,a_{\ell})\cdot P=(a',0,\ldots,0).
\]
It follows that $g\circ\Phi$ is homotopic to the composite
$
\bigvee^{\ell}_{i=1}S^{n-1}_i\overset{\text{pinch}}{\to}S^{n-1}_1\overset{a'\epsilon}{\to}A$.
On the other hand, consider the diagram of cofibration sequences
\[
\xymatrix{
\bigvee^{\ell}_{i=1}S^{n-1}\ar[r]^-{g\circ\Phi}\ar[d]^-{\Phi}	&A\ar[r]\ar@{=}[d]	&C_{g\circ\Phi}\ar[d]^-{\tilde{\Phi}}\\
\bigvee^{\ell}_{i=1}S^{n-1}\ar[r]^-{g}	&A\ar[r]	&B
}
\]
where $C_{g\circ\Phi}$ is the mapping cone of $g\circ\Phi$ and $\tilde\Phi$ is an induced map. The Five Lemma and the Whitehead Theorem imply that $\tilde{\Phi}$ is a homotopy equivalence, so $B\simeq C_{g\circ\Phi}\simeq\bigvee^{\ell}_{i=2}S^n\vee B'$, where $B'$ is the mapping cone of $a'\epsilon$.
\end{proof}
Recall that $T=\bigoplus^c_{j=1}\Z/p^{r_j}_j$ denotes the torsion group in~(\ref{table_original M hmlgy}). 
\begin{lemma}\label{lemma_decompose M/Y}
Suppose $T=0$. If $\mathcal{P}^1$ acts trivially on $H^*(M;\Z/3)$, then $C_{\jmath}\simeq_{(\frac{1}{2})}\bigvee^b_{i=1}S^4\vee S^6$. Otherwise $C_{\jmath}\simeq\bigvee^{b-1}_{i=1}S^4\vee C'$ where $C'$ is a CW-complex with one 2-cell, one 4-cell and one 6-cell.

Suppose $T\cong\Z/m$ and $\mathcal{P}^1$ acts trivially on $H^*(M;\Z/3)$. Then $C_{\jmath}\simeq\bigvee^{b-1}_{i=1}S^4\vee C''$ where $C''=P^4(m)\cup e^4\cup e^6$ has one 6-cell and 5-skeleton $P^4(m)\cup e^4$.
\end{lemma}

\begin{proof}
Assume $T=0$ and $\mathcal{P}^1$ acts trivially on $H^*(M;\Z/3)$. Then we realize case A, with $Y$ being the 3-skeleton of $M$. This gives $C_{\jmath} \simeq_{(\frac{1}{2}) }\left(\bigvee^b_{i=1}S^4\right)\cup e^6$. Since $\pi_5(\bigvee^b_{i=1}S^4)\cong\bigoplus^b_{i=1}\pi_5(S^4)\cong_{(\frac{1}{2})}0$, we have $C_{\jmath}\simeq_{(\frac{1}{2})}\bigvee^b_{i=1}S^4\vee S^6$.

Now assume $\mathcal{P}^1$ acts non-trivially on $H^*(M;\Z/3)$. Then $Y=\bigvee^{b}_{i=2}S^2_i$, and the 5-skeleton $(C_{\jmath})_5$ of $C_\jmath$ is formed by attaching $b$ 4-cells to $S^2$. By Lemma~\ref{lemma_mapping cone simplified} $(C_{\jmath})_5\simeq\tilde{C}\vee\bigvee^{b-1}_{i=1}S^4$ where $\tilde{C}$ is a CW-complex with one 2-cell and one 4-cell. The 6-cell of $C_\jmath$ is attached by a map $\hat{f}:S^5\to(C_{\jmath})_5$, whose homotopy class belongs to
\[\begin{array}{r l l}
\pi_5(\tilde{C}\vee\bigvee^{b-1}_{i=1}S^4)
&\cong&\pi_4(\Omega(\tilde{C}\vee\bigvee^{b-1}_{i=1}S^4))\\[8pt]
&\cong&\pi_4(\Omega\tilde{C})\oplus\pi_4(\Omega(\bigvee^{b-1}_{i=1}S^4))\oplus\pi_4(\Omega\tilde{C}*\Omega(\bigvee^{b-1}_{i=1}S^4))\\[8pt]
&\cong&\pi_4(\Omega\tilde{C})\oplus\bigoplus^{b-1}_{i=1}\pi_4(\Omega S^4)\oplus\pi_4(\Omega\tilde{C}*\Omega(\bigvee^{b-1}_{i=1}S^4))\\[8pt]
&\cong_{(\frac{1}{2})}&\pi_4(\Omega\tilde{C})\oplus\pi_4(\Omega\tilde{C}*\Omega(\bigvee^{b-1}_{i=1}S^4)).
\end{array}
\]
Since $\Omega\tilde{C}*\Omega(\bigvee^{b-1}_{i=1}S^4)$ is 4-connected, $\pi_4(\Omega\tilde{C}*\Omega(\bigvee^{b-1}_{i=1}S^4))$ is trivial. Thus the attaching map~$\hat{f}$ factors through some map $g:S^4\to\tilde{C}$, so $C_{\jmath}\simeq\bigvee^{b-1}_{i=1}S^4\vee C'$ where $C'$ is the mapping cone of $g$ and is a CW-complex with one 2-cell, one 4-cell and one 6-cell.

Lastly, suppose $T\cong\Z/m$ and $\mathcal{P}^1$ acts trivially on $H^*(M;\Z/3)$. Then $Y=\bigvee^b_{i=1}S^2\vee P^3(m)$ and $(C_{\jmath})_5$ is formed by attaching $b$ 4-cells to $P^4(m)$. We can follow the above argument to show that $C_{\jmath}\simeq\bigvee^{b-1}_{i=1}S^4\vee C''$ where $C''=P^4(m)\cup e^4\cup e^6$.
\end{proof}

The last tool we need is a way to understand gauge groups of principal bundles over a wedge sum $S^4\vee A$. Let $G$ be a simply connected simple compact Lie group. Then we have
\[
[S^4\vee A,BG]\cong[S^4,BG]\times[A,BG]\cong\Z\times[A,BG].
\]
A principal $G$-bundle over $A$ is classified by an element $\gamma\in[A,BG]$, and its gauge group is denoted by $\G_{\gamma}(A,G)$. Similarly, a principal $G$-bundle over $S^4\vee A$ is classified by a pair~$(l,\gamma')\in\Z\times[A,BG]$, and its gauge group is denoted by $\G_{l,\gamma'}(S^4\vee A,G)$.

\begin{lemma}\label{lemma_gauge gp over wedge sum}
Let $G$ be a simply connected simple compact Lie group and let $A$ be a space. For $\gamma\in[A,BG]$, there is a homotopy equivalence $\G_{0,\gamma}(S^4\vee A,G)\simeq\G_{\gamma}(A,G)\times\Omega^4G$.
\end{lemma}

\begin{proof}
Let $\tilde{\jmath}:A\to S^4\vee A$ be the inclusion and $\tilde{q}:S^4\vee A\to A$ the pinch map. Then $\tilde{\jmath}^*(\ell,\gamma)=\gamma$ and $\tilde{q}^*(\gamma)=(0,\gamma)$. Thus there is a diagram of evaluation fibrations
\[\xymatrix{
G\ar[r]\ar@{=}[d]	&\map^*_{\gamma}(A,BG)\ar[r]\ar[d]^-{\tilde{q}^*}	&\map_{\gamma}(A,BG)\ar[r]^-{ev}\ar[d]^-{\tilde{q}^*}	&BG\ar@{=}[d]\\
G\ar[r]				&\map^*_{0,\gamma}(S^4\vee A,BG)\ar[r]		&\map_{0,\gamma}(S^4\vee A,BG)\ar[r]^-{ev}		&BG
}\]
where the two $\tilde{q}^*$'s are induced by $\tilde{q}$. Expand the left square to get a diagram of homotopy fibration sequences
\[
\xymatrix{
	&\Omega\map^*_{0,\gamma}(S^4\vee A,BG)\ar@{=}[r]\ar[d]		&\Omega\map^*_{0,\gamma}(S^4\vee A,BG)\ar[d]\\
\G_{\gamma}(A,G)\ar[r]\ar@{=}[d]	&\G_{0,\gamma}(S^4\vee A,G)\ar[r]^-{a}\ar[d]	&\Omega^4G\ar[d]^-{s}\\
\G_{\gamma}(A,G)\ar[r]				&G\ar[r]\ar[d]				&\map^*_{\gamma}(A,BG)\ar[d]^-{\tilde{q}^*}\\
	&\map^*_{0,\gamma}(S^4\vee A,BG)\ar@{=}[r]		&\map^*_{0,\gamma}(S^4\vee A,BG)
}\]
where $s$ is some map. There is an induced map $\tilde{\jmath}^*:\map^*_{0,\gamma}(S^4\vee A,BG)\to \map_\gamma^*(A,BG)$ which satisfies $\tilde{\jmath}^*\circ \tilde{q}^*=(\tilde{q}\circ\tilde\jmath)^*\simeq id$. Because $s\simeq \tilde{\jmath}^*\circ \tilde{q}^*\circ s$, and $s$ and $\tilde{q}^*$ are consequtive maps in a homotopy fibration sequence, the map $s$ is null-homotopic. Because the middle right-hand square of the diagram is a homotopy pullback, it follows now that $\G_{0,\gamma}(S^4\vee A,G)\simeq\G_{\gamma}(A,G)\times\Omega^4G$.
\end{proof}

Combine Theorem~\ref{thm_gauge gp general decomp}, Lemmas~\ref{lemma_decompose M/Y} and~\ref{lemma_gauge gp over wedge sum} to get the following theorem.

\begin{thm}\label{thm_gauge gp M/Y}
Let $M$ be a closed, orientable, simply connected 6-manifold with homology as in~(\ref{table_original M hmlgy}), and let $P$ be a principal $SU(n)$-bundle over $M$ with $c_2(P)=(0,\ldots,0)$ and $c_3(P)=l$. Suppose $H_*(M)$ has no torsion. If $\mathcal{P}^1:H^2(M;\Z/3)\to H^6(M;\Z/3)$ is trivial, then
\[
\textstyle
\G_{\vec{0},l}(M,SU(n))\simeq_{(\frac{1}{2})}\G_{l}(S^6,SU(n))\times\prod^{b}_{i=1}(\Omega^4G\times\Omega^2SU(n))\times\prod^{2d}_{k=1}\Omega^3SU(n).
\]
If $\mathcal{P}^1:H^2(M;\Z/3)\to H^6(M;\Z/3)$ is non-trivial, then
\[
\textstyle
\G_{\vec{0},l}(M,SU(n))\simeq_{(\frac{1}{2})}\G_{\vec{0},l}(C',SU(n))\times\prod^{b-1}_{i=1}(\Omega^4G\times\Omega^2SU(n))\times\prod^{2d}_{k=1}\Omega^3SU(n),
\]
where $C'$ is a CW-complex with one 2-cell, one 4-cell and one 6-cell.

Suppose the torsion group $T\cong\Z/m$ and $\mathcal{P}^1:H^2(M;\Z/3)\to H^6(M;\Z/3)$ is trivial. Then
\begin{eqnarray*}
\G_{\vec{0},l}(M,SU(n))
&\simeq_{(\frac{1}{2})}&\textstyle\G_{\vec{0},l}(C'',SU(n))\times\prod^{b-1}_{i=1}\Omega^4G\times\prod^b_{i=1}\Omega^2SU(n)\times\\[6pt]
&&\Omega^3\{m\}SU(n)\times\textstyle\prod^{2d}_{k=1}\Omega^3SU(n)
\end{eqnarray*}
where $C''$ is a CW-complex with one 6-cell and 5-skeleton $P^4(m)\cup e^4$.
\end{thm}

\begin{proof}
Suppose $H_*(M)$ has no torsion and $\mathcal{P}^1:H^2(M;\Z/3)\to H^6(M;\Z/3)$ is trivial. By Theorem~\ref{thm_gauge gp general decomp} we have $\G_{\vec{0},l}(M,SU(n))\simeq_{(\frac{1}{2})}\G_{\vec{0},l}(C_{\jmath},SU(n))\times\prod^{b}_{i=1}\Omega^2SU(n)\times\prod^{2d}_{k=1}\Omega^3SU(n)$. In this case Lemma~\ref{lemma_decompose M/Y} implies $C_{\jmath}\simeq\bigvee^b_{i=1}S^4\vee S^6$, and Lemma~\ref{lemma_gauge gp over wedge sum} gives
\[
\textstyle
\G_{\vec{0},l}(C_{\jmath})\simeq\G_{l}(S^6)\times\prod^b_{i=1}\Omega^4SU(n).
\]
Thus we obtain the first homotopy equivalence. The other two are shown similarly.
\end{proof}


\end{document}